\documentclass[a4paper,11pt]{article}
\usepackage[plainpages=false]{hyperref}
\usepackage{amsfonts,amsmath,amssymb,amsthm,enumerate}
\usepackage{latexsym,lscape,rawfonts,mathrsfs,mathtools}
\usepackage{enumitem,pifont}
\usepackage{cases}
\usepackage{array,tabularx,tabu}

\usepackage{setspace}

\usepackage[titletoc,title]{appendix}

\usepackage{txfonts}

\newcommand{\C}{\ensuremath{\mathbb{C}}}

\newcommand{\R}{\ensuremath{\mathbb{R}}}

\newcommand{\ba}{\begin{align*}}
\newcommand{\ea}{\end{align*}}
\newcommand{\na}{\nabla}
\newcommand{\la}{\langle}
\newcommand{\ra}{\rangle}
\newcommand{\lc}{\left(}
\newcommand{\rc}{\right)}

\newcommand{\ep}{\epsilon}

\newcommand{\ka}{K\"ahler\,}

\newcommand{\tRc}{\text{Rc}}
\newcommand{\CPN}{\mathbb{C}\text{P}}

\makeatletter
\newcommand*\owedge{\mathpalette\@owedge\relax}
\newcommand*\@owedge[1]{%
\mathbin{%
\ooalign{%
$#1\m@th\bigcirc$\cr
\hidewidth$#1\m@th\wedge$\hidewidth\cr
}%
}%
}
\makeatother

\makeatletter
\def\ExtendSymbol#1#2#3#4#5{\ext@arrow 0099{\arrowfill@#1#2#3}{#4}{#5}}

\makeatother

\makeatletter
\def\ExtendSymbol#1#2#3#4#5{\ext@arrow 0099{\arrowfill@#1#2#3}{#4}{#5}}

\makeatother

\def\Xint#1{\mathchoice
{\XXint\displaystyle\textstyle{#1}}%
{\XXint\textstyle\scriptstyle{#1}}%
{\XXint\scriptstyle\scriptscriptstyle{#1}}%
{\XXint\scriptscriptstyle\scriptscriptstyle{#1}}%
\!\int}
\def\XXint#1#2#3{{\setbox0=\hbox{$#1{#2#3}{\int}$ }
\vcenter{\hbox{$#2#3$ }}\kern-.55\wd0}}

\def\aint{\Xint-}

\numberwithin{equation}{section}

\newtheorem{thm}{Theorem}[section]

\newtheorem{prop}[thm]{Proposition}
\newtheorem{lem}[thm]{Lemma}
\newtheorem{conj}[thm]{Conjecture}

\newtheorem{rem}[thm]{Remark}

\newtheorem{defn}[thm]{Definition}

\newtheorem{notn}[thm]{Notation}

\newtheorem{conv}[thm]{Convention}

\title{Rigidity of complex projective spaces in Ricci shrinkers}
\author{Yu Li \quad and \quad Wenjia Zhang}
\date{\today}

\begin{document}
\maketitle

\begin{abstract}
In this paper, we prove that any Ricci shrinker that is sufficiently close to $(\CPN^N,g_{FS})$ in the Gromov-Hausdorff sense must itself be isometric to $(\CPN^N,g_{FS})$.
\end{abstract}

\tableofcontents

\section{Introduction}
A Ricci shrinker $(M^n, g, f)$ is a complete Riemannian manifold $(M^n,g)$ together with a smooth function $f: M \to \mathbb R$ such that
\begin{equation} 
Rc+\na^2 f=\frac{1}{2}g,
\label{E100}
\end{equation}
where the potential function $f$ is normalized by adding a constant such that
\begin{align} 
(4 \pi)^{-\frac{n}{2}} \int e^{-f} \,dV_g= 1 
\label{E101}.
\end{align}

The Ricci shrinkers play essential roles in studying the singularities of the Ricci flow. In dimension $2$ or $3$, all Ricci shrinkers are completely classified (cf.~\cite{Ha95}\cite{Naber}\cite{NW}\cite{CCZ}, etc). It turns out that $\R^2,S^2,\R^3,S^3,S^2\times \R$, and their quotients are the only examples. 

In higher dimensional cases, much less is known. One strategy is to consider all Ricci shrinkers as one moduli space $\mathcal M$ equipped with pointed-Gromov-Hausdorff topology, where one can choose a minimum point of the potential function $f$ as the base point. Under the natural non-collapsing condition, the moduli space $\mathcal M$ has the weak-compactness, in the sense that any sequence of Ricci shrinkers with uniform entropy bound, by taking a subsequence, will converge to a Ricci shrinker limit space. In dimension 4, it was proved by Haslhofer-Müller \cite{HM11,HM15} that the limit is a smooth Ricci shrinker orbifold. In the general dimension, it was proved by Li-Li-Wang \cite{LLW21}, and Huang-Li-Wang \cite{HLW21} that the limit space is Ricci shrinker conifold.

One can raise the natural question: \textit{What kind of Ricci shrinker is isolated in $\mathcal M$?}

In other words, the above question discusses the characterization of the rigidity of Ricci shrinkers. Here, a Ricci shrinker is rigid if there is no nearby Ricci shrinker other than itself. Rigid examples include spherical space-forms $S^n/\Gamma$ \cite{Hui85}, $\CPN^{2m}$ by Kr\"oncke \cite{Kr16}, and $S^2 \times S^2$ proved by Sun-Zhu \cite{SZ21} recently. For other rigid compact symmetric spaces, see \cite{BHMW21} and those with $\lambda^{-1} \mu_{\text{fns}}>2$ and H. stable in \cite[Table 1, Table 2]{CH15}. For noncompact Ricci shrinkers, the rigidity problem is much more involved. Some noncompact rigid examples include $\R^n$, proved by Yokota \cite{Yo09}\cite{Yo12} (see also Li-Wang \cite{LW20}), $S^{n-1} \times \R$ proved by Li-Wang \cite{LW21}, and $S^{n-k} \times \R^k$ by Colding-Minicozzi \cite{CM21b} lately.

In this paper, we will only consider the rigidity of compact Ricci shrinkers and prove that the complex projective space $\CPN^N$ with standard metric is rigid.

\begin{thm}[Main Theorem]
For any integer $N \ge 1$, there exists a small constant $\ep=\ep(N)>0$ satisfying the following
property.

Suppose $(M^n, g, f)$ is a Ricci shrinker such that
\begin{align}
d_{GH} \left\{ (M^n,g), (\mathbb C \emph{P}^N,g_{FS})\right\}<\epsilon,
\label{E103}
\end{align} 
then $(M^n,g)$ is isometric to $(\mathbb C \emph{P}^N,g_{FS})$. Here, $n=2N$ and $g_{FS}$ is the Fubini-Study metric with Einstein constant $1/2$.
\label{T101}
\end{thm}

The proof of Theorem \ref{T101} relies on the weak-compactness theory of Ricci shrinkers developed by Li-Li-Wang \cite{LLW21} and the deformation theory of Ricci shrinkers by Podest\`a-Spiro \cite{PS15} and Kr\"oncke \cite{Kr16}. We sketch the proof as follows.

It follows from \cite[Theorem 1.1]{LLW21} that any Ricci shrinker $(M^n, g, f)$ satisfying \eqref{E103} for sufficiently small $\ep$ is close to $(\CPN^N,g_{FS})$ in the Sobolev space $H^s$ for large $s$. In particular, $M$ is diffeomorphic to $\CPN^N$ and $g$ is in a small $H^s$ neighborhood of $g_{FS}$. Therefore, one only needs to show that $g=\phi^* g_{FS}$ for a self-diffeomorphism $\phi$ of $\CPN^N$. In other words, $(\CPN^N,g_{FS})$ is rigid in the sense of \cite[Definition 4.1]{Kr16}. Notice that if we further assume $(M^n, g, f)$ is a \ka Ricci shrinker, then the conclusion already follows from the classification of \ka Ricci shrinker with positive bisectional curvature, see \cite[Theorem 3(i)]{N05}.

From \cite[Proposition 2.2]{PS15}, there exists a slice $\mathcal{S}^s_{f}$ at $g_{FS}$ such that for a small $H^s$ neighborhood $\mathcal{U}$ of $g_{FS}$, any $g \in \mathcal{U}$ is isometric to a unique metric $\tilde g \in \mathcal{S}^s_{f}$. Moreover, $ \mathcal{S}^s_{f}$ is a smooth manifold with tangent space $T_{g_{FS}} \mathcal{S}^s_{f} = \text{ker}(\delta_{g_{FS}})$. In addition, it follows from \cite[Theorem 3.4]{PS15} that all Ricci shrinkers in $\mathcal{S}^s_{f}$ forms a real-analytic subset of a finitely-dimensional real-analytic submanifold.

Suppose $(\CPN^N,g_{FS})$ is not rigid, then there exists a smooth curve $g_t$ of Ricci shrinkers in $\mathcal{S}^s_{f}$ starting from $g_{FS}$. In particular, $g_t$ satisfies the Ricci shrinker equation \eqref{E100} at $t=0$ up to any order. Our strategy is  to show that this cannot happen. More precisely, we will prove that any infinitesimal solitonic deformation of $(\CPN^N,g_{FS})$ is not integrable of at most third-order. Notice that if $N=2m$, Kr\"oncke \cite[Theorem 6.1]{Kr16} has already proved that $(\CPN^{2m},g_{FS})$ is not integrable of second-order. Thus, we will focus on the case $N=2m-1$, and prove that $(\CPN^{2m-1},g_{FS})$ is not integrable of third-order and hence is rigid. We remark that even though $ (\CPN^N,g_{FS})$ is rigid, it is dynamically unstable under the Ricci flow \cite[Corollary 1.8]{Kr20}.

This paper is organized as follows. Section $2$ recalls Ricci shrinkers' deformation theory, which will be used throughout the paper. Moreover, we determine all possible infinitesimal solitonic deformations. In section $3$, we obtain all possible second-order deformations by solving the second-order Ricci shrinker equation. Combined with the first-order deformation, we prove that none is integrable of third-order, and hence the main theorem is proved. In the last section, we discuss the rigidity of the product Ricci shrinker with the complex projective space as a factor.

{\bf Acknowledgements}: 
Yu Li is supported by YSBR-001, NSFC-12201597 and research funds from USTC (University of Science and Technology of China) and CAS (Chinese Academy of Sciences). Both authors would like to thank Prof. Bing Wang for his interest in this work.

\section{Preliminaries}
%%%%geometry of moduli space
Let $(M^{n}, \bar g, \bar f )$ be a fixed Ricci shrinker on a closed manifold $M^n$ satisfying \eqref{E100} and \eqref{E101} and $s$ is an integer with $s>n/2+3$.

From Sun-Wang \cite[Lemma 2.2]{SW15}, there exists a $H^s$ neighborhood $\mathcal{U}$ of $\bar g$, such that for any $g \in \mathcal{U}$, the minimizer of $\boldsymbol{\mu} (g,1)$ is unique and depends real-analytically on $g$. Here, $\boldsymbol{\mu} (g,1)$ is the celebrated functional defined by Perelman \cite{Pe1}. In particular, if we denote the minimizer by $f=f(g)$, then $f$ satisfies
\begin{align*}
(4 \pi)^{-\frac{n}{2}} \int e^{-f} \,dV_g= 1 
\end{align*} 
and the Euler-Lagrange equation
\begin{align*}
2 \Delta f - |\nabla f| ^2 + R(g) + f - n = \boldsymbol{\mu} (g,1).
\end{align*}
On $\mathcal U$, we have the following definition.
\begin{defn}
The Ricci shrinker operator is defined as
\begin{align*}
\Phi(g) : = \frac{1}{2} g - Rc(g) - \na^2_g f
\end{align*}
for any $g \in \mathcal U$. In particular, $\Phi(g)=0$ if and only if $(M,g,f)$ is a Ricci shrinker.
\end{defn}

Next, we recall the following notations.

\begin{notn}
Let $(M^n,g,f)$ be a Riemannian manifold coupled with a smooth function $f$. 
\begin{enumerate}[label=(\roman*)]
\item $\mathcal S^p(M)$ consists of smooth symmetric $p$-forms.

\item The weighted divergence $\delta_{f} : C^{\infty} (\mathcal{S}^{p}(M)) \rightarrow C^{\infty}(\mathcal{S}^{p-1}(M))$ is defined as
\begin{align*}
(\delta_{f} T)(X_1,\cdots,X_{p-1}) = \sum_{i=1}^{n} \nabla_{e_i} T(e_i,X_1,\cdots,X_{p-1}) - T(\na f,X_1,\cdots,X_{p-1})
\end{align*} 
for any $T \in \mathcal{S}^{p}(M)$, where $\{e_i\}$ is an orthonormal basis.

\item $\delta_{f}^{*}$ is the formal adjoint of $\delta_f$ with respect to the form $e^{-f} dV_g$. More precisely, 
\begin{align*}
(\delta_{f}^{*} T)(X_1,\cdots,X_p) = -\frac{1}{p} \sum_{i=0}^{p-1} \nabla_{X_{1+i}} T (X_{2+i},\cdots,X_{p+i})
\end{align*} 
for any $T \in \mathcal{S}^{p-1}(M)$, where the sums $1+i,\cdots,p+i $ are taken modulo $p$.
\end{enumerate}
If $f$ is a constant, then we omit $f$ in $\delta_f$ and $\delta_{f}^{*}$.
\end{notn}

For $(M^{n}, \bar g, \bar f )$, we recall the decomposition $S^2 (M) = \text{ker}(\delta_{\bar f}) \oplus \text{im}(\delta_{\bar f}^{*})$. Moreover, one has the following slice-theorem proved in \cite[Proposition 2.2]{PS15}, which can be regarded as a generalization of the classic Ebin's slice theorem \cite{EB70}.

\begin{prop} \label{Prop:201}
There exists a submanifold $\mathcal{S}_{\bar f}^s$ of all $H^s$ metrics satisfying the following properties:
\begin{enumerate}[label=\alph*)]
\item There is a small $H^s$ neighborhood $\mathcal{U'}$ of $\bar g$ in the set of metrics such that any $g \in \mathcal{U'}$ is isometric to a unique metric $\tilde g \in \mathcal{S}^s_{\bar f}$.

\item $ \mathcal{S}^s_{\bar f}$ is a smooth manifold with tangent space $T_{\bar g} \mathcal{S}^s_{\bar f} = \emph{ker}(\delta_{\bar f})$.
\end{enumerate}
\end{prop}

By Proposition \ref{Prop:201}, one may focus on the slice $\mathcal{S}^s_{\bar f}$ and all Ricci shrinkers nearby on the slice can be represented as 
\begin{align*}
\mathscr{S}:= \left\{ g \in \mathcal{S}^s_{\bar f} \mid \Phi(g)=0 \right\}.
\end{align*} 
Moreover, it follows from \cite[Theorem 3.4]{PS15} that $\mathscr S$ is a real-analytic subset of a finitely-dimensional real-analytic submanifold of $\mathcal{S}^s_{\bar f}$. Therefore, if $(M,\bar g, \bar f)$ is not rigid, then there exists a nontrivial smooth curve $g_t \subset \mathscr S$ with $g_0=\bar g$.

For any integer $k \ge 1$, from
\begin{align*}
\partial_{t^k}^k \rvert_{t=0} \Phi(g_t)=0
\end{align*} 
we have
\begin{align} \label{E201}
\Phi'(g^{(k)}) + \sum_{l = 2}^{k} \sum_{1 \leq k_1 \leq \cdots \leq k_l , k_1+ \cdots + k_l = k} C(k,l,k_1,\cdots,k_l) \Phi^{(l)} (g^{(k_1)},\cdots,g^{(k_l)}) =0
\end{align} 
where $g^{(i)}=\partial_{t^i}^i \rvert_{t=0} g_t \in S^2(M)$ and $\Phi^{(i)}$ is the $i$-th variation of $\Phi$. Moreover, $C(k,l,k_1,\cdots,k_l)$ are constants depending only on $k,l,k_1,\cdots,k_l$. In particular, we have 
\begin{align*}
g' \in \text{ker}(\Phi') \cap \text{ker}(\delta_{\bar f})=:\text{ISD}
\end{align*} 
where the last notation stands for infinitesimal solitonic deformation.

Conversely, we have the following definition.
\begin{defn}
For any nontrivial $h = g^{(1)} \in \emph{ISD}$. We call $h$ integrable up to order $k$, if there exists a sequence of tensors $g^{(2)}, \cdots,g^{(k)} \in S^2 (M)$ satisfying \eqref{E201}.
\end{defn}

Notice that by \cite[Lemma 5.2]{Kr16}, if $h$ is integrable up to any order, then there exists a smooth curve $g_t \in \mathscr{S}$ with $\partial_t \rvert_{t=0} g_t=h$. In particular, $(M,\bar g, \bar f)$ is not rigid.

In this paper, we are only interested in the special case that $(M,\bar g, \bar f)$ is Einstein, that is, $\bar f$ is constant. We first recall the following first variation of $\Phi$.

\begin{lem} \label{lem:201}
For any $h \in S^2(M)$
\begin{equation} \label{E202a}
2\Phi'(h) = \Delta h + 2 Rm(h) + 2 \delta^{*}\delta h + \nabla^2 (H-2 f')
\end{equation}
where $H = \emph{Tr}(h)$ and $f'$ is determined by
\begin{equation} \label{E202b}
( \Delta+\frac 1 2 ) (H-2f')=\delta^2 h.
\end{equation}
Moreover, $\Phi'$ is a self-adjoint operator on $S^2(M)$ preserving the decomposition $S^2 (M) = \emph{ker}(\delta) \oplus \emph{im}(\delta^{*})$. In particular, for $h \in \emph{ker}(\delta) $, we have
\begin{equation*}
2\Phi'(h) = \Delta h+ 2Rm(h) \quad \text{and} \quad H=2f'.
\end{equation*}
\end{lem}
\begin{proof}
The proofs of formulae \eqref{E202a} and \eqref{E202b} can be found, e.g., in \cite[Lemma 2.3]{CZ12} and \cite[(2.15)]{CZ12} respectively. Notice that $H-2f'$ is uniquely determined since the first eigenvalue of $\Delta$ is greater than $n/2(n-1)$ by Lichnerowicz \cite{Lic58}. Moreover, it is immediately known from \eqref{E202a} and \eqref{E202b} that $\Phi'$ is a self-adjoint operator.

Since $\Phi(\bar g)=0$, we conclude that $\Phi(\phi^*\bar g)=0$ for any self-diffeomorphism $\phi$ of $M$. In particular, $\text{im}(\delta^*)\subset \text{ker}(\Phi')$. In addition, if $h \in \text{ker}(\delta)$, one has $2\Phi'(h) = \Delta h+ 2Rm(h)$ and hence
\begin{align*}
2\delta(\Phi'(h))=\delta(\Delta h+ 2Rm(h))=(\Delta+\frac 1 2)(\delta h)=0.
\end{align*} 
In sum, the proof is complete.
\end{proof}

From Lemma \ref{lem:201}, it is easy to obtain the following decomposition, see \cite[Lemma 6.2]{Kr14} for details.

\begin{lem} \label{lem:202}
We have
\begin{align*}
\emph{ISD} = \emph{IED} \oplus \{ u \bar g + 2 \nabla^2 u \mid u \in C^{\infty}(M), \Delta u + u =0 \},
\end{align*} 
where \emph{IED} consists of all TT \emph{(traceless-transverse)} $h \in S^2(M)$ such that $ \Delta h+ 2Rm(h)=0$.
\end{lem}

Now, we focus on the case $(M,\bar g)=(\CPN^N,g_{FS})$. It follows immediately from \cite{Ko80} that there is no nontrivial infinitesimal Einstein deformation. In other words, $\text{IED}=0$. Therefore, we only need to consider the conformal variation. Next, we recall the following result of Kröncke \cite[Theorem 5.7]{Kr16}.

\begin{prop} \label{prop:202}
Let $(M, \bar g)$ be an Einstein manifold with Einstein constant $1/2$. Let $u \in C^{\infty}(M)$ be such that $\Delta u + u = 0$. Then $u \bar g + 2 \nabla^2 u \in \emph{ISD}$ is not integrable of second-order if there exists another function $w \in C^{\infty}(M)$ with $\Delta w + w = 0$ such that 
$$\int_{M} u^2 w \neq 0.$$
\end{prop}

\begin{rem}
In Proposition \ref{prop:202}, if we further assume $\emph{IED}=0$, then $u \bar g + 2 \nabla^2 u$ is not integrable of second-order \textbf{only if} there exists another function $w$ with $\Delta w + w = 0$ such that 
$$\int_{M} u^2 w \neq 0.$$
\end{rem}

By using Proposition \ref{prop:202}, Kröncke \cite[Theorem 6.1]{Kr16} proved the rigidity of $(\CPN^{2m},g_{FS})$:

\begin{thm} \label{T:201}
All infinitesimal solitonic deformations of $(\mathbb C \emph{P}^{2m},g_{FS})$ are not integrable of second-order. Therefore, $(\mathbb C \emph{P}^{2m},g_{FS})$ is rigid.
\end{thm}

Notice that if $N=2m-1$, there indeed exists $u$ with $\Delta u+u=0$ such that $\int u^2 w=0$ for any $w$ with $\Delta w+w=0$. In fact, we have the following precise characterization of $u$.

\begin{lem} \label{lem:203}
On $(\CPN^{2m-1},g_{FS})$, if there exists a function $u$ with $\Delta u+u=0$ such that
\begin{align*}
\int u^2 w =0
\end{align*} 
for any $w$ with $\Delta w+w=0$. Then, after a possible change of coordinates, 
\begin{align}
u([z_0,z_1,\cdots,z_{2m-1}]) = \lambda \lc \frac{|z_0|^2+ \cdots+ |z_{m-1}|^2 - |z_{m}|^2 - \cdots- |z_{2m-1}|^2}{|z_0|^2 + \cdots +|z_{2m-1}|^2 } \rc \label{E202}
\end{align} 
for a constant $\lambda$.
\end{lem}
\begin{proof}
Let $P_{k,k}$ be the space of polynomials on $\C^{2m}$ which are homogeneous of degree k in $z$ and $\bar{z}$, and
let $H_{k,k} $ the subspace of harmonic polynomials in $P_{k,k}$. It is well-known that $P_{k,k} = H_{k,k} \oplus r^2 P_{k-1,k-1}$ and any function $u$ with $\Delta u+u=0$ on $\CPN^{2m-1}$ can be lifted to a function $f \in H_{1,1}$. 

By a change of coordinates, we may assume $f = \sum \lambda_{i} |z_{i}|^2 $ and $\sum \lambda_{i}=0$. Since
\begin{align*}
f^2 \in P_{2,2} = H_{2,2} \oplus r^2 H_{1,1} \oplus \R \cdot r^4,
\end{align*} 
it follows from our assumption that $f^2 \in H_{2,2} \oplus \R \cdot r^4$ and hence $\Delta f^2 \in \R \cdot \Delta (r^4)$. On the one hand,
\begin{align} \label{E204a}
\Delta f^2=8 \sum \lambda_i^2 |z_i|^2.
\end{align} 
On the other hand, we have
\begin{align} \label{E204b}
\Delta r^4=8(2m+1)r^2.
\end{align} 
Comparing \eqref{E204a} with \eqref{E204b}, we conclude that $|\lambda_i|$ are the same and the conclusion follows.
\end{proof}

\begin{conv}
For later computations, we require
\begin{enumerate}[label=\alph*)]
\item The scalar curvature $R$, the operator $\Delta$, etc., are the concepts in real Riemannian geometry.

\item The curvature operator in local coordinates are given by $R_{abcd} = \la [\nabla_{\partial_a},\nabla_{\partial_b}] \partial_c - \nabla_{[\partial_a,\partial_b]} \partial_c ,\partial_d \ra$. Therefore, for any $h \in S^2$, $Rm(h)_{ac}=-R_{abcd}h_{bd}$.

\item $i,j,k,\cdots$ and $\bar i, \bar j,\bar k,\cdots$ denote the local holomorphic and antiholomorphic coordinates respectively, unless otherwise stated.
\end{enumerate}
\end{conv}

\section{Proof of the main theorem}

Throughout this section, we consider $(M^n,g)=(\CPN^{2m-1},g_{FS})$ for $m \ge 2$. Notice that $\CPN^1=S^2$ and its rigidity is obvious. We will do most calculations in the domain $U_0 = \{ z_0=1 | z \in \CPN^{2m-1} \}$, i.e., $(z_1,\cdots,z_{2m-1})$ are the local coordinates. In addition, the function $u$ is defined as in \eqref{E202}.

\subsection*{Basic equations for $u$}

\begin{notn} For simplicity, we define
\begin{align*}
A : =1+ |z_1|^2 + \cdots + |z_{m-1}|^2,\quad B : = |z_m|^2 + \cdots+ |z_{2m-1}|^2,\quad S:=A+B.
\end{align*} 
\end{notn}

Under the local coordinates, the Fubini-Study metric and its inverse are represented as
\begin{equation}
g_{i \bar{j}} =4m \lc \frac{\delta_{ij}}{S} - \frac{\bar{z}_{i} z_{j}}{ S^{2}} \rc \quad \text{and} \quad g^{i \bar j}=\frac{S}{4m} \lc \delta_{ij}+z_i \bar z_j \rc.
\label{E321}
\end{equation}
Moreover, its Christoffel coefficients are
\begin{align}\label{E321a}
\Gamma_{ij}^k=-\frac{1}{S} \lc \delta_{ik} \bar z_j+\delta_{jk} \bar z_i \rc.
\end{align} 
In the following computation, we also use the lower index to denote the corresponding covariant derivatives. For instance, 
 $u_i,u_{\bar j},u_{ij},u_{i \bar j},u_{\bar i \bar j}$ denote $\partial_{z_i} u, \partial_{\bar z_j} u, \na^2 u(\partial_{z_i},\partial_{z_j}),\na^2 u(\partial_{z_i},\partial_{\bar z_j}),\na^2 u(\partial_{\bar z_i},\partial_{\bar z_j})$ respectively.

Next, we have the following results by direct computation
\begin{lem} \label{lem:301}
For the function $u$ defined in \eqref{E202}, we have
\begin{align*}
\lambda^{-1} u_i=
\begin{cases}
2S^{-2}B \bar{z}_{i}, \quad &\forall i<m, \\
-2S^{-2}A \bar{z}_{i}, \quad &\forall i \ge m,
\end{cases}
\quad \text{and} \quad
u_{\bar i}=\overline{u_i}.
\end{align*} 
\end{lem}
\begin{proof}
From our definition, $\lambda^{-1}u=2S^{-1}A-1$ and hence
\begin{align*}
\lambda^{-1}u_i=2S^{-1}A_i-2S^{-2}A \bar z_i.
\end{align*} 
From this, the conclusion follows.
\end{proof}

Next, we compute the Hessian of $u$.

\begin{lem} \label{lem:302}
We have
\begin{align*}
\lambda^{-1} u_{i \bar j}=
\begin{cases}
2S^{-2}B (\delta_{ij}-2 S^{-1} \bar z_i z_j ), \quad &\forall i,j<m, \\
&\\
-2S^{-2}A(\delta_{ij}-2 S^{-1} \bar z_i z_j), \quad &\forall i,j\ge m, \\
&\\
2S^{-3}(A-B) \bar z_i z_j,\quad &\text{otherwise}.
\end{cases}
\end{align*} 
Morever, $u_{ij}=u_{\bar i \bar j}=0$.
\end{lem}

\begin{proof}
We may assume $\lambda=1$. By direct calculations,
\begin{align*}
\partial_j \partial_i u=\partial_j(2S^{-1}A_i-2S^{-2}A\bar z_i)=4S^{-3}A\bar z_i z_j-2S^{-2}A_i \bar z_j-2S^{-2}A_j \bar z_i.
\end{align*} 
Therefore, by \eqref{E321a},
\begin{align*}
u_{ij}=\partial_j \partial_i u-\Gamma_{ij}^k u_k=\partial_j \partial_i u-S^{-1}(\delta_{ik}\bar z_j+\delta_{jk} \bar z_i)(2S^{-1}A_k-2S^{-2}A\bar z_k)=0.
\end{align*} 
Since $u$ is a real function, $u_{\bar i \bar j}=0$ as well. In addition, 
\begin{align*}
u_{i \bar j}=&\partial_{\bar j} \partial_i u=\partial_{\bar j}(2S^{-1}A_i-2S^{-2}A\bar z_i) \\
=& -2S^{-2}S_{\bar j} A_i+2S^{-1}A_{i \bar j}+4S^{-3}AS_{\bar j}\bar z_i-2S^{-2}A_{\bar j}\bar z_i-2S^{-2}A\delta_{ij} \\
=& -2S^{-2}A_i z_j+2S^{-1}A_{i \bar j}+4S^{-3}A\bar z_i z_j-2S^{-2}A_{\bar j}\bar z_i-2S^{-2}A\delta_{ij}.
\end{align*} 
From this, the conclusion follows immediately.
\end{proof}

\begin{lem} \label{lem:303}
$u$ satisfies the following identities:
\begin{enumerate}[label=(\alph*)]
\item $|\nabla u|^2 = \dfrac{1}{2m}(\lambda^2 - u^2)$.

\item $\Delta u^2 = \dfrac{\lambda^2}{m} - (2+\dfrac{1}{m})u^2$.

\item $\displaystyle \aint u^2 = \dfrac{\lambda^2}{2m+1}$, where $\displaystyle \aint =\frac{1}{\emph{Vol}(M)} \int$.

\item $\displaystyle \aint u^4 = \dfrac{3\lambda^4}{(2m+1)(2m+3)}$.
\end{enumerate}
\end{lem}

\begin{proof}
Without loss of generality, we assume $\lambda = 1$.

(a) We compute from \eqref{E321} and Lemma \ref{lem:301} that
\begin{align*}
&|\nabla u|^2 = 2g^{i \bar j}u_i u_{\bar j} \\
=& \frac{S}{2m} \lc \sum_{i,j <m} (\delta_{ij}+z_i \bar z_j) 4S^{-4}B^2 \bar z_i z_j+ \sum_{i,j \ge m} (\delta_{ij}+z_i \bar z_j) 4S^{-4}A^2 \bar z_i z_j+\sum_{\text{otherwise}} z_i\bar z_j (-4S^{-4}AB\bar z_i z_j)\rc \\
=& \frac{2S^{-3}}{m} \lc B^2A(A-1)+ A^2(B^2+B)-2A(A-1)B^2\rc =\frac{2S^{-2}AB}{m}=\frac{1}{2m}(1-u^2).
\end{align*} 

(b) From (a), we have
\begin{equation*}
\Delta u^2= 2u \Delta u + 2 |\nabla u|^2 = -2 u^2 + \frac{1}{m} (1 - u^2) = \frac{1}{m} - (2+\frac{1}{m})u^2.
\end{equation*}

(c) Integrating (b), we get
\begin{align*}
(2+\frac{1}{m})\int u^2=\frac{1}{m} \text{Vol}(M),
\end{align*} 
and the identity follows.

(d) We compute
\begin{align*}
\Delta u^4=& 2 (\Delta u^2) u^2 + 2 |\nabla u^2|^2 \\
=& 2(\frac{1}{m} - (2+\frac{1}{m})u^2)u^2 + 8 u^2 (\frac{1}{2m}(1 - u^2)) \\
=& \frac{6}{m} u^2 + (-4-\frac{6}{m})u^4.
\end{align*} 
By integration, we have
\begin{align*}
\aint u^4=\frac{3}{2m+3} \aint u^2=\frac{3}{(2m+1)(2m+3)}.
\end{align*} 
\end{proof}

From now on, we will only consider the infinitesimal deformation $ug$ by Lemma \ref{lem:202} and Lemma \ref{lem:203}. Notice that we consider here $u g$ instead of $ug + 2 \nabla^2 u$ as one can compose with a family of diffeomorphisms generated by $-\na u$.

Our strategy goes as the following:

\textbf{Step 1}: Solve all possible second variation $g^{(2)}$ from
\begin{equation}
\Phi^{(2)}(ug,u g) + \Phi'(g^{(2)}) = 0.
\label{E301}
\end{equation}

%\item Then we will prove that $h$ is integrable up to the third-order if and only if 
%$$\phi'''(ug,ug,ug) + 3 \phi''(ug,g^{(2)}) \perp ker(\Phi'). $$

\textbf{Step 2}: There exists a symmetric 2-tensor $\eta \in \text{ker}( \Phi')$ such that for any $g^{(2)}$ obtained in Step 1,
\begin{equation}
\int_{M} \la \Phi^{(3)}(ug,ug,ug) + 3 \Phi^{(2)}(ug,g^{(2)}),\eta \ra \neq 0.
\label{E311}
\end{equation}

Combining these two steps, one can prove that $ug$ is not integrable of third-order.

\subsection*{Second Variation}
Now, we set $g(t)=(1+tu)g$ and $f(t)=f(g(t))$. Moreover, we denote the partial derivative at $t=0$ by the subscript $t$.

By Lemma \ref{lem:A02}(i),
\begin{equation}
\Phi^{(2)}(ug,ug) = -\nabla^2 f_{tt} -2(m-1) du \otimes du+ (u^2 - |\nabla u |^2) g -4(m-1) u \nabla^2 u,
\label{E303}
\end{equation}
where $f_{tt}$ is determined by
\begin{align}
(\Delta+\frac{1}{2})f_{tt} =& (2m-1) u^2 - (3m-2) |\nabla u|^2 \notag \\
=& (2m-\frac{1}{m}+ \frac{1}{2})u^2 + (\frac{1}{m}-\frac{3}{2})\lambda^2.
\label{E302}
\end{align}
Here, we have used Lemma \ref{lem:303} (a).

\begin{lem} \label{lem:304}
The solution $f_{tt}$ of \eqref{E302} is
\begin{equation*}
f_{tt}=-\frac{4m^2+m-2}{3m+2}u^2-\frac{m-2}{3m+2} \lambda^2.
\end{equation*}
\end{lem}

\begin{proof}
We choose two functions $u^2,\lambda^2$ as basis to form a linear space and calculate from Lemma \ref{lem:303} (b) that 
\begin{equation*}
\Delta 
\begin{bmatrix}
u^2 & \lambda^2 
\end{bmatrix}=
\begin{bmatrix}
u^2 & \lambda^2 
\end{bmatrix}
\begin{bmatrix}
-(2+\frac{1}{m}) & 0\\
\frac{1}{m} &0
\end{bmatrix}.
\end{equation*}
and hence
\begin{equation*}
(\Delta + \frac{1}{2}) 
\begin{bmatrix}
u^2 & \lambda^2 
\end{bmatrix}=
\begin{bmatrix}
u^2 & \lambda^2 
\end{bmatrix}
\begin{bmatrix}
-(\frac{3}{2}+\frac{1}{m}) & 0 \\
\frac{1}{m} & \frac{1}{2}
\end{bmatrix}.
\end{equation*}
In particular, the operator $\Delta+1/2$ preserves the linear space $\la u^2,\lambda^2 \ra$. Since the kernel of $\Delta+1/2$ is trivial, we have from \eqref{E302}
\begin{equation*}
\begin{aligned}
f_{tt} 
&= 
\begin{bmatrix}
u^2&\lambda^2
\end{bmatrix}
\begin{bmatrix}
-(\frac{3}{2}+\frac{1}{m}) & 0 \\
\frac{1}{m} & \frac{1}{2}
\end{bmatrix}^{-1}
\begin{bmatrix}
2m-\frac{1}{m}+\frac{1}{2}\\ \frac{1}{m} - \frac{3}{2}
\end{bmatrix}\\
&= 
\begin{bmatrix}
u^2&\lambda^2
\end{bmatrix}
\begin{bmatrix}
-\frac{2m}{3m+2} & 0 \\
\frac{4}{3m+2} & 2
\end{bmatrix}
\begin{bmatrix}
2m-\frac{1}{m}+\frac{1}{2}\\ \frac{1}{m} - \frac{3}{2}
\end{bmatrix} \\
& = 
\begin{bmatrix}
u^2& \lambda^2
\end{bmatrix}
\begin{bmatrix}
-\frac{4m^2+m-2}{3m+2}\\ -\frac{m-2}{3m+2}
\end{bmatrix}.
\end{aligned}
\end{equation*}
\end{proof}

From \eqref{E301} and \eqref{E303}, we try to solve $h$ in 
\begin{equation*}
\Phi'(h) = \nabla^2 f_{tt} +2 (m-1)du \otimes du + (|\nabla u|^2- u^2) g+4(m-1) u \nabla^2 u
\end{equation*}
Since $\text{im}(\delta^*) \subset \text{ker}(\Phi')$ and $\Phi'$ preserves $\text{ker}(\delta)$, we may further assume $\delta h=0$. In particular, from Lemma \ref{lem:201}, Lemma \ref{lem:303} (a) and Lemma \ref{lem:304}, one has
\begin{equation}
\begin{aligned}
\frac{1}{2}Lh & = \nabla^2 f_{tt} +2 (m-1)du \otimes du + (|\nabla u|^2- u^2) g+4(m-1) u \nabla^2 u \\
% & = 2a (du \otimes du + u \nabla^2 u) - u^2 g + (2m-2)du \otimes du + |\nabla u|^2 g+(4m-4) u \nabla^2 u \\
& = \frac{\lambda^2}{2m}g -\frac{2m+1}{2m} u^2 g -\frac{2m(m+2)}{3m+2}du \otimes du +\frac{4m^2-6m-4}{3m+2} u \nabla^2 u,
\end{aligned}
\label{E307}
\end{equation}
where, for simplicity, we denote the operator $\Delta+2Rm$ by $L$.

Our next goal is to obtain a particular solution $h_0$ of \eqref{E307}. To achieve this, we first recall a general result for \ka manifolds.

\begin{lem} \label{lem:305}
On any \ka manifold, the operator $L$ satisfies
\begin{equation}
L(S^{(1,1)}) \subset S^{(1,1)} \quad \text{and} \quad L(S^{(0,2)} \oplus S^{(2,0)}) \subset S^{(0,2)} \oplus S^{(2,0)},
\label{E330}
\end{equation}
where we denote the space of symmetric $(p,q)$-form by $S^{(p,q)}$.
\end{lem}
\begin{proof}
Recall the fact that $h \in S^{(1,1)}$ if and only if $h(X,Y)=h(JX,JY)$, and $h \in S^{(0,2)} \oplus S^{(2,0)}$ if and only if $h(X,Y) = - h(JX,JY)$. By slightly abusing the notation, we define an operator $J: S^{2} \rightarrow S^2 $,where $J(h) (X,Y) = h(JX,JY)$.

To show \eqref{E330}, one only needs to prove
\begin{equation}
L(Jh) = J(Lh).
\label{E332}
\end{equation}
Under an orthonormal basis $\{e_i\}$, we compute
\begin{align*}
Rm(Jh)(e_a,e_b)=&-Rm(e_a,e_i,e_b,e_i) h(Je_i,Je_i) \\
=&-Rm(e_a,Je_i,e_b,Je_i) h(e_i,e_i) \\
=&-Rm(Je_a,e_i,Je_b,e_i) h(e_i,e_i)=J(Rm(h))(e_a,e_b).
\end{align*} 
Combined with the fact that $J$ is parallel, we readily conclude \eqref{E332} and thus \eqref{E330} is proved.
\end{proof}

For later calculation, we define the following auxiliary tensor.

\begin{defn}
We define a symmetric $(1,1)$-form $\xi$, where
$$\xi_{i \bar{j}}= \xi_{\bar{j} i} = g^{\bar{k} l} u_{i \bar{k}} u_{l \bar{j}}$$
\end{defn}

One important property of $\xi$ is that it can be represented as a linear combination of $\{\lambda^2 g, u^2 g, \partial u \otimes \bar{\partial} u + \bar{\partial} u \otimes \partial u, u \nabla^2 u \} $. 

\begin{lem} \label{lem:306}
We have
\begin{align*}
\xi_{i \bar{j}} = \frac{\lambda^2 - u^2}{16m^2} g_{i \bar{j}} - \frac{1}{4m} u_{i} u_{\bar{j}} - \frac{1}{2m} u u_{i \bar{j}}.
\end{align*} 
\end{lem}
\begin{proof}

Without loss of generality, we assume $\lambda=1$.

For $i,j<m$, it follows from Lemma \ref{lem:302} that
\begin{align*}
\xi_{i \bar{j}}=&g^{\bar k l} u_{i \bar k} u_{l \bar j}\\
=& \sum_{k,l <m} 4S^{-4}B^2 g^{\bar k l} ( \delta_{ik}-2 S^{-1} \bar z_i z_k)(\delta_{lj}-2 S^{-1} \bar z_l z_j) \\
&+\sum_{k<m,l \ge m} 4S^{-5}B(A-B) g^{\bar k l} ( \delta_{ik}-2 S^{-1} \bar z_i z_k) \bar z_l z_j\\
&+\sum_{k\ge m,l < m} 4S^{-5}B(A-B) g^{\bar k l} ( \delta_{lj}-2 S^{-1} \bar z_l z_j) \bar z_i z_k\\
&+ \sum_{k,l \ge m} 4S^{-6}(A-B)^2g^{\bar k l} \bar z_i z_k \bar z_l z_j \\
& =:I_{11}+I_{12}+I_{13}+I_{14}.
\end{align*}

We compute
\begin{align*}
I_{11}=&m^{-1}B^2S^{-3} \sum_{k,l <m} (\delta_{kl}+\bar z_k z_l)( \delta_{ik}-2S^{-1} \bar z_i z_k)(\delta_{lj}-2 S^{-1}\bar z_l z_j) \\
=& m^{-1}B^2S^{-3} \sum_{l <m} \lc \delta_{il}+(1-2S^{-1}A)\bar z_i z_l \rc (\delta_{lj}-2 S^{-1}\bar z_l z_j) \\
=& m^{-1}B^2S^{-3} \lc \delta_{ij}+(1-4S^{-1}A+4S^{-2}A^2-4S^{-2}A)\bar z_i z_j \rc.
\end{align*}

Moreover,
\begin{align*}
I_{12}=I_{13}=& m^{-1}(A-B)BS^{-4} \sum_{k \ge m, l<m} z_l \bar z_k \bar z_i z_k(\delta_{lj}-2S^{-1} \bar z_l z_j) \\
=&m^{-1}(A-B)B^2S^{-4} \sum_{l<m} z_l\bar z_i (\delta_{lj}-2S^{-1} \bar z_l z_j)\\
=&m^{-1}(A-B)B^2S^{-4} (1+2S^{-1}-2S^{-1}A) \bar z_i z_j
\end{align*}
and we have
\begin{align*}
I_{14}=& m^{-1}(A-B)^2S^{-5} \sum_{k ,l \ge m} (\delta_{kl}+z_l \bar z_k) \bar z_i z_k \bar z_l z_j =m^{-1}(A-B)^2B(1+B)S^{-5} \bar z_i z_j.
\end{align*}

Combining all terms, we obtain for $i,j<m$,
\begin{align*}
\xi_{i \bar j}=m^{-1}B^2S^{-3}\lc \delta_{ij}+(B^{-1}-4S^{-1}) \bar{z}_{i} z_{j} \rc.
\end{align*}

In the same way, we can calculate the other cases and obtain
\begin{align*}
\xi_{i \bar j}=
\begin{cases}
m^{-1}B^2S^{-3}\lc \delta_{ij}+(B^{-1}-4S^{-1}) \bar{z}_{i} z_{j} \rc, \quad &\forall i,j<m, \\
&\\
m^{-1}A^2S^{-3}\lc \delta_{ij}+(A^{-1}-4S^{-1}) \bar{z}_{i} z_{j} \rc, \quad &\forall i,j\ge m, \\
&\\
-m^{-1}(A-B)^2S^{-4} \bar z_i z_j,\quad &\text{otherwise}.
\end{cases}
\end{align*} 

By using Lemma \ref{lem:301} and Lemma \ref{lem:302}, one can easily check
\begin{align*}
\xi_{i \bar{j}} = \frac{\lambda^2 - u^2}{16m^2} g_{i \bar{j}} - \frac{1}{4m} u_{i} u_{\bar{j}} - \frac{1}{2m} u u_{i \bar{j}}
\end{align*} 
and we omit the detailed calculations.
\end{proof}

Now we can completely solve $h$ in \eqref{E307}.

\begin{thm}
The solution set of \eqref{E307}, denoted by $\mathcal{H} $, can be represented as
\begin{align*}
\mathcal{H} =h_0+K_0 \oplus \emph{im}(\delta^*),
\end{align*} 
where $K_0:= \{ v g + 2 \nabla^2 v \mid v \in C^{\infty}(M), \Delta v + v =0 \}$ and $h_0$ is a particular solution in $\emph{ker}(\delta)$ defined as
\begin{align}
h_0:=&-\frac{2} {m+1} \lambda^2 g + \frac{2m}{m+1} u^2 g+ \frac{4m(m^2+5m+2)}{(m+1)(3m+2)} (\partial u \otimes \bar{\partial} u + \bar{\partial} u \otimes \partial u) \notag \\
&-\frac{8m^3}{(m+1)(3m+2)} u \na^2u +\frac{4m^2(m+2)}{(m+1)(3m+2)} (\partial u \otimes \partial u + \bar{\partial} u \otimes \bar{\partial} u). \label{E333a}
\end{align} 
\end{thm}

\begin{proof}
From Lemma \ref{lem:201} and Lemma \ref{lem:202}, we only need to prove $h_0$ is indeed a particular solution in $\text{ker}(\delta)$.

Observe that the right side of $\eqref{E307}$ is in the linear space $V$ spanned by $\{\lambda^2 g, u^2 g, \partial u \otimes \bar{\partial} u + \bar{\partial} u \otimes \partial u, u \nabla^2 u, \partial u \otimes \partial u + \bar{\partial} u \otimes \bar{\partial} u \} $. 
So we consider solving an explicit $h_0$ in this linear space. We will show that $L$ actually preserves $V$, and hence one can solve a particular $h_0$ in $V$ by inverting $L$.

In the following, we do calculations of $L$ on each basis.

\begin{enumerate}[label=(\alph*)]
\item We compute 
\begin{align*}
L(g) = \Delta(g) + 2 Rm(g) = 2 Rc = g.
\end{align*} 

\item From Lemma \ref{lem:303} (b), we have
\begin{align*}
L(u^2 g) = (\Delta u^2)g + 2 u^2 Rm (g) 
= \lc \frac{\lambda^2}{m} - (2+\frac{1}{m})u^2 \rc g + u^2 g 
= \frac{\lambda^2}{m} g - (1+\frac{1}{m}) u^2 g.
\end{align*} 

\item Notice that under general coordinates $a,b,c$, etc., we have by Bochner's formula
\begin{align*}
\Delta u_a=(\Delta u)_a+g^{bc}R_{ab}u_c=-u_a+\frac{1}{2}u_a=-\frac{1}{2} u_a.
\end{align*} 
Therefore,
\begin{align} \label{E334a}
\Delta(u_a u_b) = (\Delta u_a) u_b + (\Delta u_b) u_a + 2g^{cd}u_{ac} u_{bd}=- u_a u_b + 2 \xi_{ab}.
\end{align} 

From Lemma \ref{lem:305}, $Rm(\partial u \otimes \bar{\partial} u + \bar{\partial} u \otimes \partial u)$ is of type $(1,1)$.
More precisely, we recall
\begin{equation*}
R_{i \bar{j} k \bar{l}} = \frac{1}{4m}(g_{i \bar{j}} g_{k \bar{l}} + g_{i \bar{l}} g_{k \bar{j}})
\end{equation*}
and hence
\begin{align} \label{E334b}
Rm(u_i u_{\bar{j}})_{k \bar{l}} = - R_{k \bar{t} \bar{l} s} u_i u_{\bar{j}} g^{i \bar{t}} g^{\bar{j} s} 
= \frac{1}{4m} (u_k u_{\bar{l}} + \frac{1}{2} |\nabla u|^2 g_{k \bar{l}})= \frac{1}{4m}u_k u_{\bar{l}} +\frac{\lambda^2-u^2}{16m^2} g_{k \bar l},
\end{align} 
where we have used Lemma \ref{lem:303} (a).

Combined with \eqref{E334a}, we obtain
\begin{align*}
L(u_i u_{\bar{j}})_{k \bar l}
& = \Delta(u_i u_{\bar{j}})_{k \bar l} + 2 Rm (u_i u_{\bar{j}})_{k \bar l}\\
& = - u_k u_{\bar{l}} + 2 \xi_{k \bar{l}} + \frac{1}{2m}u_k u_{\bar{l}} +\frac{\lambda^2-u^2}{8m^2} g_{k \bar l}\\
& = -u_k u_{\bar{l}} - \frac{1}{m} u u_{k \bar{l}} + \frac{\lambda^2-u^2}{4m^2} g_{k \bar{l}},
\end{align*} 
where for the last line we have used Lemma \ref{lem:306}.

\item Similarly, we have
\begin{equation*}
Rm(u_i u_j)_{k l} = - R_{\bar{t} k \bar{s} l} u_i u_j g^{i \bar{t}} g^{j \bar{s}} = -\frac{1}{2m} u_k u_l
\end{equation*}
and hence
\begin{align*}
L(u_i u_j)_{k l}
& = \Delta(u_i u_j)_{k l} + 2 Rm (u_i u_{j})_{k l}=-(1+\frac{1}{m}) u_ku_l.
\end{align*} 

\item Next, we compute
\begin{align} \label{E334c}
L(u \nabla^2 u) = \Delta u \nabla^2 u + 2 \la \nabla u, \nabla(\nabla^2 u) \ra + u L(\na^2 u)=2 \la \nabla u, \nabla(\nabla^2 u) \ra-u \na^2u,
\end{align} 
where we have used the fact that $L(\na^2 u)=\na^2(\Delta u+u)=0$. Moreover,
\begin{align*}
2g^{cd}u_d( \na_c \na_a \na_b u)=&2g^{cd}u_d( \na_a \na_b \na_c u)-2g^{cd}g^{ef}R_{cabe} u_f u_d \\
=& \na_a \na_b (g^{cd}u_cu_d)-2g^{cd}\na_a \na_c u \na_b \na_d u-2g^{cd}g^{ef}R_{cabe} u_f u_d.
\end{align*} 
In other words,
\begin{align*}
2 \la \nabla u, \nabla(\nabla^2 u) \ra=\na^2 |\na u|^2-2\xi-2Rm(du \otimes du).
\end{align*} 
Combined with \eqref{E334c}, we obtain
\begin{align*}
L(u \nabla^2 u)_{k \bar l}=& |\na u|^2_{k \bar l}-2\xi_{k \bar l}-2Rm(du \otimes du)_{k \bar l}-u u_{k \bar l} \\
=& -\frac{\lambda^2 -u^2}{4m^2} g_{k \bar{l}} -\frac{1}{m} u_k u_{\bar{l}} - u u_{k \bar{l}},
\end{align*} 
where we have used Lemma \ref{lem:306}, Lemma \ref{lem:303} (a) and \eqref{E334b}.
\end{enumerate}

To summarize, the matrix of $L$ acts on the basis
$\{\lambda^2 g, u^2 g, \partial u \otimes \bar{\partial} u + \bar{\partial} u \otimes \partial u, u \nabla^2 u, \partial u \otimes \partial u + \bar{\partial} u \otimes \bar{\partial} u \} $
by
\begin{equation*}
\begin{aligned}
L 
& \begin{bmatrix}
\lambda^2 g_{i\bar{j}} & u^2 g_{i\bar{j}} & u_i u_{\bar{j}} & u u_{i \bar{j}} & u_i u_j 
\end{bmatrix} \\
& =
\begin{bmatrix}
\lambda^2 g_{i\bar{j}} & u^2 g_{i\bar{j}} & u_i u_{\bar{j}} & u u_{i \bar{j}} & u_i u_j 
\end{bmatrix}
\begin{bmatrix}
1 & \frac{1}{m} & \frac{1}{4m^2} & -\frac{1}{4m^2} & 0\\ 
0 & -1-\frac{1}{m} & -\frac{1}{4m^2} & \frac{1}{4m^2} & 0\\ 
0 & 0 & -1 & -\frac{1}{m} & 0\\ 
0 & 0 & -\frac{1}{m} & -1 & 0\\ 
0 & 0 & 0 & 0 & -\frac{1}{m}-1
\end{bmatrix}
\end{aligned}.
\end{equation*}

The matrix of $L^{-1}$ is 
\begin{equation*}
\begin{aligned}
\begin{bmatrix}
1 & \frac{1}{m+1} & \frac{1}{4(m^2-1)} & -\frac{1}{4(m^2-1)} & 0\\
0 & -\frac{m}{m+1} & \frac{1}{4(m^2-1)} & -\frac{1}{4(m^2-1)} & 0\\
0 & 0 & \frac{m^2}{1-m^2} & \frac{m}{m^2 -1} & 0\\
0 & 0 & \frac{m}{m^2 -1} & \frac{m^2}{1-m^2} & 0\\
0 & 0 & 0 & 0 & -\frac{m}{m+1}
\end{bmatrix}
\end{aligned}.
\end{equation*}
Therefore, we can solve one explicit $h_0$ from \eqref{E307} by 
\begin{equation*}
\begin{aligned}
\frac{1}{2} h_0 
&= 
\begin{bmatrix}
\lambda^2 g_{i\bar{j}} & u^2 g_{i\bar{j}} & u_i u_{\bar{j}} & u u_{i \bar{j}} & u_i u_j 
\end{bmatrix}
\begin{bmatrix}
1 & \frac{1}{m+1} & \frac{1}{4(m^2-1)} & -\frac{1}{4(m^2-1)} & 0\\
0 & -\frac{m}{m+1} & \frac{1}{4(m^2-1)} & -\frac{1}{4(m^2-1)} & 0\\
0 & 0 & \frac{m^2}{1-m^2} & \frac{m}{m^2 -1} & 0\\
0 & 0 & \frac{m}{m^2 -1} & \frac{m^2}{1-m^2} & 0\\
0 & 0 & 0 & 0 & -\frac{m}{m+1}
\end{bmatrix}
\begin{bmatrix}
\frac{1}{2m} \\ -1-\frac{1}{2m} \\ -\frac{2m(m+2)}{3m+2} \\
\frac{4m^2-6m-4}{3m+2} \\ -\frac{2m(m+2)}{3m+2}
\end{bmatrix} \\
& =
\begin{bmatrix}
\lambda^2 g_{i\bar{j}} & u^2 g_{i\bar{j}} & u_i u_{\bar{j}} & u u_{i \bar{j}} & u_i u_j 
\end{bmatrix}
\begin{bmatrix}
-\frac{1} {1+m}\\
\frac{m}{m+1} \\ 
\frac{2m(m^2+5m+2)}{(1+m)(3m+2)} \\
-\frac{4m^3}{(m+1)(3m+2)} \\ 
\frac{2m^2(m+2)}{(m+1)(3m+2)}
\end{bmatrix}.
\end{aligned}
\end{equation*}

It remains to check that $h_0 \in \text{ker}(\delta)$. Indeed, one can easily compute
\begin{align*}
\begin{cases}
& \delta(\lambda^2 g)=0,\\
&\delta(u^2 g) = 2u \nabla u, \\
& \delta \lc\partial u \otimes \bar{\partial} u + \bar{\partial} u \otimes \partial u \rc=-\frac{1}{2} u \nabla u, \\
& \delta(u \nabla^2 u) =\frac{1}{2} \na |\na u|^2+u\delta(\na^2 u)= - (\frac{1}{2m} + \frac{1}{2}) u \nabla u,\\
& \delta \lc \partial u \otimes \partial u + \bar{\partial} u \otimes \bar{\partial} u \rc=- (\frac{1}{2m} + \frac{1}{2}) u \nabla u.
\end{cases}
\end{align*} 
Therefore, it is easy to see from \eqref{E333a} that 
\begin{align*}
\delta(h_0)=\lc \frac{4m}{m+1}-\frac{2m(m^2+5m+2)}{(1+m)(3m+2)}+\frac{4m^2}{3m+2}-\frac{2m(m+2)}{3m+2} \rc u \nabla u=0.
\end{align*} 

In sum, the proof is complete.
\end{proof}

\subsection*{Third-order obstruction}

In the subsection, we prove that 
\begin{equation}\label{E323}
I(h):=\int_{M} \la \Phi^{(3)} (ug,ug,ug) + 3 \Phi^{(2)} (h, ug ) , ug \ra \neq 0
\end{equation}
for any $h \in \mathcal H$. In other words, we choose $\eta=ug$ in \eqref{E311}.

Recall that $\mathcal{H} =h_0+K_0 \oplus \text{im}(\delta^*)$, where $h_0$ is defined in \eqref{E333a}. We next show that $I(h)$ depends only on $I(h_0)$.

\begin{prop} \label{prop:301}
With above definitions, for any $h \in \mathcal H$,
\begin{align*}
I(h)=I(h_0).
\end{align*} 
\end{prop}

The proof relies on the fact that $\int \la 3 \Phi^{(2)} (h, ug ) , ug \ra$ depends only on $h_0$, which the following two lemmas can show.

\begin{lem} \label{lem:321}
For any $1$-form $\alpha$, we have
\begin{equation*}
\int \la \Phi^{(2)}(ug,\delta^{*}\alpha),ug \ra = 0. 
\end{equation*}
\end{lem}
\begin{proof}
Let $X$ be the vector field dual to $\alpha$ and $\phi_s$ be a family of diffeomorphisms generated by $X$.

Evaluated at $t=s=0$, we have
\begin{align} \label{E323a}
\partial^2_{st} \Phi \lc \phi_{s}^{*}((1+tu)g) \rc=\partial^2_{st}\Phi(g + t ug + s L_{X} g)+ \Phi'(L_{X} (ug))=\Phi^{(2)}(ug, L_{X} g) + \Phi' \lc L_{X} (ug) \rc
\end{align} 
since 
\begin{align*}
\phi_{s}^{*}((1+tu)g) = (1 + tu) g + s L_{X} \lc (1+ t u) g \rc + R_1= (1 + tu) g + s L_{X} g + st L_{X} (ug) + R_2,
\end{align*} 
where $R_1$ and $R_2$ are remainders containing only terms of $t^2$ or $s^2$ or higher-order.

On the other hand, we have
\begin{align} \label{E323b}
\partial^2_{st} \int \la \Phi \lc \phi_{s}^{*}((1+tu)g) \rc,ug \ra_g\,dV_g=\partial^2_{st} \int \la \Phi((1+tu)g),\phi_{-s}^*(ug) \ra_{\phi_{-s}^*(g)}\,dV_{\phi_{-s}^*(g)}=0
\end{align} 
since $\Phi'(ug)=0$. Combining \eqref{E323a} and \eqref{E323b}, we obtain
\begin{align*}
&-2\int \la \Phi^{(2)}(ug,\delta^{*}\alpha),ug \ra=\int \la \Phi^{(2)}(ug,L_{X} g),ug \ra \\
=&- \int \la \Phi'(L_{X} (ug)),ug \ra 
= - \int \la L_{X} (ug), \Phi'(ug) \ra 
= 0,
\end{align*}
where we have used the fact that $\Phi'$ is self-adjoint.

\end{proof}

\begin{lem} \label{lem:322}
For any smooth function $v$ with $\Delta v+v=0$, we have
\begin{align*}
\int \la \Phi^{(2)}(ug, vg),ug \ra = 0. 
\end{align*} 
\end{lem}
\begin{proof}
We set $g(t,s)=(1+tu+sv)g$. Then it follows from Lemma \ref{lem:A03} \eqref{EA06}, \eqref{EA07} that
\begin{align} \label{E324a}
\la \Phi_{st},g \ra=(4-6m) \la \na u,\na v \ra+2(4m-3)uv-\Delta f_{st},
\end{align} 
where $f_{st}$ satisfies
\begin{align} \label{E324b}
(\Delta +\frac 1 2) f_{st}=(2m-1)uv-(3m-2) \la \na u,\na v \ra.
\end{align} 
Multiplying \eqref{E324b} by $u$ on both sides and integrating, we obtain
\begin{align} \label{E324c}
\int (2m-1)u^2v-\frac{3m-2}{2} \la \na u^2,\na v \ra = \int u (\Delta +\frac 1 2) f_{st} =-\frac{1}{2} \int u f_{st} .
\end{align}

From \eqref{E324a} and \eqref{E324c}, we have
\begin{align*}
\int \la \Phi_{st},ug \ra =& \int (2-3m) \la \na u^2,\na v \ra+2(4m-3)u^2v-u\Delta f_{st} \\
=& \int (2-3m) u^2v+2(4m-3)u^2v-(4m-2)u^2v+(3m-2) \la \na u^2,\na v \ra \\
=& 4(m-1) \int u^2 v=0,
\end{align*} 
where the last equality follows from our choice of $u$ in Lemma \ref{lem:203}.
\end{proof}

Combining Lemma \ref{lem:321} and Lemma \ref{lem:322}, Proposition \ref{prop:301} follows immediately.

To obtain \eqref{E323}, we only need to show $I(h_0) \ne 0$. For simplicity, we define
\begin{align*}
I_1: = \int_{M} \la \Phi^{(3)} (ug,ug,ug) , ug \ra \quad \text{and} \quad I_2 : =\int_{M} \la 3 \Phi^{(2)} (h_0, ug ) , ug \ra.
\end{align*} 

\begin{prop} \label{prop:303}
With above definitions, we have
\begin{align*}
I_1=-\frac{6(4m^3-3m^2+3m-2)}{(2m+1)(2m+3)(3m+2)}\emph{Vol}(M) \lambda^4.
\end{align*} 
\end{prop}

\begin{proof}
We define $g(t)=(1+tu)g$. By Lemma \ref{lem:A02}(ii), we have 
\begin{align*}
\Phi_{ttt} 
=& -\nabla^2 f_{ttt} + \frac{3}{2} \lc du \otimes df_{tt} + d f_{tt} \otimes du - \la \nabla f_{tt}, \nabla u \ra g \rc\\
& + 24(m-1) u du \otimes du + 12(m-1) u^2 \nabla^2 u -3 u^3 g - 6(m-2) u |\nabla u|^2 g,
\end{align*} 
where $f_{ttt}$ satisfies
\begin{equation}
\begin{aligned}
(\Delta + \frac 1 2) f_{ttt} =& 3u \Delta f_{tt} - 6(3m-2) u^3 + 9(3m-2) u |\nabla u|^2.
\end{aligned}
\label{E308}
\end{equation}

Now, we denote
\begin{align*}
I_{11} := - \int \la \nabla^2 f_{ttt}, u g \ra= -\int u \Delta f_{ttt}=\int u f_{ttt}.
\end{align*} 
From $\int u(\Delta+1)f_{ttt}=0$, Lemma \ref{lem:303} (a) (b) and Lemma \ref{lem:304}, we conclude
\begin{align}
I_{11}=& - 2 \int u (\Delta + \frac 1 2) f_{ttt} \notag \\
=&-6 \int u^2 \Delta f_{tt} - 2(3m-2) u^4 + 3(3m-2) u^2 |\nabla u|^2 \notag \\
=&-6 \int \Delta u^2\lc -\frac{4m^2+m-2}{3m+2}u^2 \rc+\frac{3(3m-2)}{2m} u^2(\lambda^2-u^2)- 2(3m-2) u^4 \notag \\
=& -6 \int \lc \frac{\lambda^2}{m}-(2+\frac{1}{m})u^2 \rc \lc -\frac{4m^2+m-2}{3m+2}u^2 \rc \notag\\
& -6 \int \frac{3(3m-2)}{2m} u^2(\lambda^2-u^2)- 2(3m-2) u^4 \notag \\
=& \frac{3(20m^3+15m^2-10m-8)}{m(3m+2)}\int u^4-\frac{3(19m^2-2m-8)}{m(3m+2)} \lambda^2 \int u^2. \label{E351a}
\end{align} 

Similarly, we compute
\begin{align}
I_{12} :=& \frac{3}{2}\int \la du \otimes df_{tt} + d f_{tt} \otimes du - \la \nabla f_{tt}, \nabla u \ra g , ug \ra \notag\\
&+ \int \la 24(m-1) u du \otimes du + 12(m-1) u^2 \nabla^2 u -3 u^3 g - 6(m-2) u |\nabla u|^2 g ,ug \ra \notag\\
=& \int 6(1-m)\la \na u,\na f_{tt} \ra u-12(2m^2-7m+4)u^2 |\na u|^2-6(4m-3)u^4\notag \\
=& \int - \frac{6(2m^3-14m^2+m+6)}{m(3m+2)} u^2 (\lambda^2-u^2)-6(4m-3)u^4\notag \\
=& -\frac{6(2m+3)(5m^2-m-2)}{m(3m+2)}\int u^4-\frac{6(2m^3-14m^2+m+6)}{m(3m+2)} \lambda^2 \int u^2. \label{E351b}
\end{align} 

Combining \eqref{E351a} and \eqref{E351b}, we have
\begin{align*}
I_1=&I_{11}+I_{12}=-\frac{3(11m^2-4m-4)}{m(3m+2)}\int u^4-\frac{3(m-2)(4m^2-m-2)}{m(3m+2)} \lambda^2 \int u^2 \\
=& \lc -\frac{9(11m^2-4m-4)}{m(2m+1)(2m+3)(3m+2)}-\frac{3(m-2)(4m^2-m-2)}{m(2m+1)(3m+2)} \rc \text{Vol}(M) \lambda^4, \\
=&-\frac{6(4m^3-3m^2+3m-2)}{(2m+1)(2m+3)(3m+2)}\text{Vol}(M) \lambda^4,
\end{align*} 
where we have used Lemma \ref{lem:303} (c) (d).
\end{proof}

Next, we compute $I_2$.

\begin{prop} \label{prop:304}
With above definitions, we have
\begin{align*}
I_2=\frac{6(4m^4+25m^3-32m^2-7m+14)}{(m+1)(2m+1)(2m+3)(3m+2)} \emph{Vol}(M) \lambda^4.
\end{align*} 
\end{prop}

\begin{proof}
We define $g(t,s)=(1+tu)g+sh_0$ and $H_0=\text{Tr}(h_0)$. Then it follows from \eqref{E333a} and Lemma \ref{lem:303} (a) that
\begin{align}
H_0=&-\frac{2(4m-2)} {m+1} \lambda^2 + \frac{2m(4m-2)}{m+1} u^2+ \frac{4m(m^2+5m+2)}{(m+1)(3m+2)} |\na u|^2 +\frac{8m^3}{(m+1)(3m+2)} u^2 \notag\\
=& -\frac{2(11m^2-3m-6)} {(m+1)(3m+2)} \lambda^2+\frac{2(16m^3+m^2-9m-2)}{(m+1)(3m+2)} u^2.
\label{E353a}
\end{align} 

Now it follows from Lemma \ref{lem:A04} \eqref{EA09c} that
\begin{align*}
\la \Phi_{st},g \ra = -2(m-1) \la h_0,\nabla^2 u \ra - \frac{3}{2} \la \nabla H_0, \nabla u \ra - u \Delta H_0 +\frac{u H_0}{2} - \Delta f_{st},
\end{align*} 
where $f_{st}$ satisfies 
\begin{align*}
(\Delta +\frac 1 2) f_{st} = -\frac{1}{2} u \Delta H_0 - \frac{3}{4} \la \nabla u , \nabla H_0 \ra.
\end{align*} 
By our definition,
\begin{align}
\frac{1}{3} I_2 & =\int \la \Phi_{st}, u g \ra= \int -2(m-1) u \la h_0,\nabla^2 u \ra - \frac{3}{2} u \la \nabla H_0, \nabla u \ra - u^2 \Delta H_0 +\frac{u^2 H_0}{2} - u \Delta f_{st} \notag \\
& =\int -2(m-1) u \la h_0,\nabla^2 u \ra - \frac{3}{4} \la \nabla H_0, \nabla u^2 \ra - H_0 \Delta u^2 +\frac{u^2 H_0}{2}+uf_{st} \notag\\
&= \int -2(m-1) u \la h_0,\nabla^2 u \ra+\frac{u^2 H_0}{2}, \label{E353b}
\end{align} 
where we have used Lemma \ref{lem:303} (b) and the identity
\begin{align*}
\int uf_{st}=-2\int u (\Delta+\frac 1 2) f_{st} =\int u^2 \Delta H_0+\frac{3}{4} \la \na u^2,\na H_0 \ra.
\end{align*}

Now, we define and compute
\begin{align*}
I_{21}:=&\int u\la h_0,\nabla^2 u \ra \\
=& \int \frac{2} {m+1} \lambda^2 u^2-\frac{2m}{m+1} u^4+ \frac{4m(m^2+5m+2)}{(m+1)(3m+2)} u\la \partial u \otimes \bar{\partial} u + \bar{\partial} u \otimes \partial u,\na^2 u \ra \\
&+\int -\frac{8m^3}{(m+1)(3m+2)} u^2 |\na^2u|^2+\frac{4m^2(m+2)}{(m+1)(3m+2)} u\la \partial u \otimes \partial u + \bar{\partial} u \otimes \bar{\partial} u, \na^2 u\ra.
\end{align*} 

Since $\na^2 u$ is of type $(1,1)$, we have
\begin{align} \label{E333b}
\la \partial u \otimes \partial u + \bar{\partial} u \otimes \bar{\partial} u, \na^2 u\ra=0
\end{align} 
and
\begin{align} 
\la \partial u \otimes \bar{\partial} u + \bar{\partial} u \otimes \partial u,\na^2 u \ra=2g^{l \bar j} g^{i \bar k}u_i u_{\bar j} u_{\bar k l}=\frac{1}{2} \la \na |\na u|^2,\na u \ra =-\frac{1}{2m} u|\na u|^2=-\frac{1}{4m^2}u(\lambda^2-u^2). \label{E333c}
\end{align} 
Moreover, it follows from Lemma \ref{lem:306} and Lemma \ref{lem:303} (a) that
\begin{align} \label{E333d}
|\na^2u|^2=&2g^{i \bar j} \xi_{i \bar j}= \frac{2m-1}{8m^2}(\lambda^2-u^2)-\frac{1}{4m} |\na u|^2+\frac{1}{2m} u^2= \frac{m-1}{4m^2}\lambda^2+\frac{m+1}{4m^2} u^2.
\end{align} 

Combining \eqref{E333b}, \eqref{E333c} and \eqref{E333d}, we have
\begin{align*}
I_{21}=& \int \frac{2} {m+1} \lambda^2 u^2-\frac{2m}{m+1} u^4- \frac{m^2+5m+2}{m(m+1)(3m+2)} u^2(\lambda^2-u^2) \\
&-\int \frac{8m^3}{(m+1)(3m+2)} u^2 \lc \frac{m-1}{4m^2}\lambda^2+\frac{m+1}{4m^2} u^2 \rc \\
=& -\frac{8m^2-3m-2}{m(3m+2)}\int u^4-\frac{2m^3-7m^2+m+2}{m(m+1)(3m+2)} \lambda^2 \int u^2
\end{align*} 

Therefore, it follows from \eqref{E353a} and \eqref{E353b} that
\begin{align*}
\frac{1}{3} I_2 =&-2(m-1)I_{21}+ \int \frac{u^2H_0}{2}  \\
=& \frac{32m^4-5m^3-29m^2+4m+4}{m(m+1)(3m+2)}\int u^4+\frac{4m^4-29m^3+19m^2+8m-4}{m(m+1)(3m+2)} \lambda^2 \int u^2 \\
=& \frac{2(4m^4+25m^3-32m^2-7m+14)}{(m+1)(2m+1)(2m+3)(3m+2)} \text{Vol}(M) \lambda^4
\end{align*} 
\end{proof}

Finally, we can prove the main theorem of the section.
\begin{thm} \label{thm:301}
For every $h \in \mathcal{H}$,
\begin{equation*}
\int \la \Phi^{(3)} (ug,ug,ug) + 3 \Phi^{(2)} (h, ug ) , ug \ra =\frac{48(m-1)^2}{(m+1)(2m+1)(2m+3)} \emph{Vol}(M) \lambda^4.
\end{equation*}
\end{thm}

\begin{proof}
From our calculations above, we have
\begin{align*}
& I(h_0)=I_1+I_2 \\
=&\lc -\frac{6(4m^3-3m^2+3m-2)}{(2m+1)(2m+3)(3m+2)}+\frac{6(4m^4+25m^3-32m^2-7m+14)}{(m+1)(2m+1)(2m+3)(3m+2)} \rc \text{Vol}(M) \lambda^4 \\
=&\frac{48(m-1)^2}{(m+1)(2m+1)(2m+3)} \text{Vol}(M) \lambda^4.
\end{align*}
\end{proof}

Now, Theorem \ref{thm:301} implies that $ug$ is not integrable of the third-order and hence $\CPN^{2m-1}$ is rigid. Combined with Theorem \ref{T:201}, Theorem \ref{T101} is proved.

\section{Further discussion}

In this section, we consider the rigidity of the product of a complex projective space of complex even dimension and an Einstein manifold $M_2$. For simplicity, we assume $(M_1,g_1)=(\CPN^{2m},g_{FS})$ and $(M_2,g_2)$ is an Einstein manifold with Einstein constant $1/2$ satisfying the following conditions:
\begin{align}
\begin{cases}
1 \notin \text{spec}(\Delta_{2}), \tag{$\dagger$} \label{E401}\\
(-\infty, -1] \cup \{ 0 \} \notin \text{spec}(L_2 \vert_{\text{TT}}).
\end{cases}
\end{align}
Here, we use the subscript $2$ to denote operators concerning $g_2$. Moreover, the fact that $L_2=\Delta_2+2Rm_2$ preserves the TT-subspace follows from the next lemma.

\begin{lem} \label{lem:400}
Let $(M,g)$ be a compact Einstein manifold with $Rc=\lambda g$ for a constant $\lambda$. Then the operator $L=\Delta+2Rm$ preserves the decomposition $S^2(M)=\emph{ker}(\delta) \oplus \emph{im}(\delta^*)$ and the TT-subspace. Moreover, the smallest eigenvalue of $L$ restricted on $\emph{im}(\delta^*)$ is at least $-\lambda$.
\end{lem}

\begin{proof}
The conclusions follow immediately from the following identities, which can be verified by direct calculations.
\begin{align*}
\delta L(h)=(\Delta+\lambda)(\delta h), \quad L(\delta^* \alpha)=\delta^*((\Delta+\lambda)\alpha) \quad \text{and} \quad \text{Tr}(Lh)=(\Delta+2\lambda)(\text{Tr}(h))
\end{align*}
for any $h \in S^2(M)$ and $1$-form $\alpha$.
\end{proof}

\begin{rem}
Many compact symmetric spaces including $S^n$ for $n \ge 3$ satisfy \eqref{E401}. We refer the readers to \emph{\cite[Table 1, Table 2]{CH15}}.
\end{rem}

Next, we find all infinitesimal solitonic deformations on the product.
\begin{lem} \label{lem:401}
For $(M,g)=(M_1\times M_2,g_1 \times g_2 )$ with condition \eqref{E401}, we have
\begin{align*}
\emph{ISD} = \{ u g_1 + 2 \nabla_1^2 u + v g_2 \, | \, u, v \in C^{\infty}(M_1), \Delta_1 u + u = \Delta_1 v + v = 0 \}.
\end{align*}
\end{lem}
\begin{proof}
From Lemma \ref{lem:201}, $2\Phi'_g=L_g$ on $\text{ker}(\delta_g)$ and hence we only need consider the kernel of $L_g$ on $S^2(M)$. It follows from \cite[Proposition 4.1]{Kr15} that
\begin{align*}
\text{spec}(L_g) = (\text{spec}(L_1)+\text{spec}(\Delta_2)) \cup (\text{spec}(L_2)+\text{spec}(\Delta_1)) \cup (\text{spec}(\Delta_1^1)+\text{spec}(\Delta_2^1)),
\end{align*}
where $\Delta^1$ denotes the Laplacian on the $1$-form. In addition, the eigentensor of $L_g$ can be expressed as the product of the corresponding eigensections. Notice the any tensor in the kernel of $L_g$ can only be the products of the first two types since $\text{spec}(\Delta_i^1)>0$ for $i=1,2$. Indeed, if $\text{spec}(\Delta_i^1)=0$, it implies that the universal cover of $M_i$ splits a line, which is impossible.

It follows from \cite{Bou10} that $\text{spec}(L_1)=\{ -1,0,\cdots \}$. Moreover, if $L_1(h_1)-h_1=0$, then $h_1=cg_1$ for some constant $c$. If $L_1(h_1)=0$, then $h_1=ug_1+2\na_1^2u$ for some function $u$ with $\Delta_1 u+u=0$. By our assumption that $1 \notin \text{spec}(\Delta_{2})$, we conclude any tensor in the kernel of $L_g$ of the first type can be represented as $ug_1+2\na_1^2u$ with $\Delta_1 u+u=0$.

Next, we take any $h=vh_2$ in the kernel of $L_g$, where $v$ is an eigenfunction of $\Delta_1$ and $h_2$ is an eigentensor of $L_2$. Since $\text{spec}(\Delta_1)=\{ 0,1,\cdots \}$, see \cite{Bou10}, we assume $\Delta_1 v+\lambda_1 v=0$ for $\lambda_1=0,1$ or $\lambda_1>1$.

If $\lambda_1=0$, then $v$ is a constant and $L_2(h_2)=0$. We decompose $h_2=\tilde h_2+\delta^*_2 \alpha \in \text{ker}(\delta_2) \oplus \text{im}(\delta_2^*)$, where $\alpha$ is a $1$-form on $M_2$. From Lemma \ref{lem:400}, we conclude that $L_2(\tilde h_2)=0$. By taking the trace of the last equation, we obtain $\Delta_2 \tilde H_2+\tilde H_2=0$, where $\tilde H_2=\text{Tr}_{g_2} \tilde h_2$. It follows from our assumption \eqref{E401} that $\tilde H_2=0$ and hence $\tilde h_2=0$. Therefore, $h=\delta^*_2 \alpha$ with $L_2(\delta^*_2 \alpha)=0$. In particular, $h \in \text{im}(\delta^*_g)$ as one can regard $\alpha$ as a $1$-form on $M$.

If $\lambda_1=1$, then $\Delta_1 v+v=0$ and $L_2(h_2)-h_2=0$. It follows from Lemma \ref{lem:400} that $h_2 \in \text{ker}(\delta_2)$. Moreover, by taking the trace, $\Delta_2 H_2=0$, where $H_2=\text{Tr}_{g_2}(h_2)$ and hence $H_2$ is a constant. In other words, $L_2(h_2-H_2g_2/n_2)-(h_2-H_2g_2/n_2)=0$, where $n_2=\text{dim}(M_2)$. Thus, by \eqref{E401}, we conclude that $h_2=H_2g_2/n_2$.

If $\lambda_1>1$, then $L_2(h_2)-\lambda_1 h_2=0$. Again, it follows from Lemma \ref{lem:400} that $h_2 \in \text{ker}(\delta_2)$. Moreover, by taking the trace, $\Delta_2 H_2+(1-\lambda_1)H_2=0$ and hence $H_2=0$. In other words, $h_2$ belongs to the TT-subspace and by \eqref{E401}, $h_2=0$ since $-\lambda_1 \notin \text{spec}(L_2 \vert_{\text{TT}})$.

From the above discussion, we conclude that any $h \in \text{ker}(L_g)$ can be written as
\begin{align*}
h=u g_1 + 2 \nabla_1^2 u + v g_2+\delta_g^* \alpha,
\end{align*}
where $u,v$ are functions on $M_1$ with $\Delta_1 u+u=\Delta_1 v+v=0$ and $\alpha$ is a $1$-form on $M$. From this, the conclusion immediately follows.
\end{proof}

Next, we prove
\begin{lem} \label{lem:402}
Let $(M^n,g)=(M^{n_1}_1\times M^{n_2}_2,g_1 \times g_2 )$ such that $(M_i,g_i)$ \emph{($i=1,2$)} is an Einstein manifold with Einstein constant $1/2$. Then
\begin{align*}
\Phi^{(2)}(ug,vg_1)=-\na^2 f_{st}-\frac{n_1-2}{4} (du \otimes dv+dv \otimes du)-g_1 \la \na u, \na v \ra +\frac{g+g_1}{2}uv-\frac{n_1-2}{2} (u \na^2 v+v \na^2 u),
\end{align*}
where $u,v$ are functions on $M_1$ with $\Delta_1 u+u=\Delta_1 v+v=0$ and $f_{st}$ is determined by
\begin{align*} 
(\Delta+\frac{1}{2}) f_{st}= \frac{n_1}{2} uv-\frac{3n_1-2}{4} \la \na u, \na v \ra.
\end{align*}
In particular, we have
\begin{align} \label{E410a}
\la \Phi^{(2)}(ug,vg_1),g \ra=-\Delta f_{st}+\frac{2-3n_1}{2} \la \na u, \na v \ra+\frac{n+3n_1-4}{2} uv
\end{align}
and 
\begin{align}\label{E410b}
\la \Phi^{(2)}(ug,vg_1),g_1 \ra=-\Delta f_{st}+\frac{2-3n_1}{2} \la \na u, \na v \ra+2(n_1-1) uv.
\end{align}
\end{lem}
\begin{proof}
We define $\tilde g(t,s)=(1+tu)(g+sv g_1)$ and $\widetilde{Rc}=Rc(\tilde g)$.

From Lemma \ref{lem:A01}(1), we have at $t=0$,
\begin{align*}
\widetilde{Rc}_t=-\frac{n-2}{2}\na^2_{g+svg_1} u - \frac{1}{2} (\Delta_{g+svg_1} u)(g+svg_1).
\end{align*}
By taking the derivative for $s$ and using Lemma \ref{lem:A01} (7) (9), we obtain
\begin{align*}
\widetilde{Rc}_{st} 
=&\frac{n-2}{4} (du \otimes dv+dv \otimes du- \la \na u, \na v \ra g_1)+\frac{uv}{2}g_1-\frac{g}{2} \lc uv+\frac{n_1-2}{2} \la \na u, \na v \ra \rc \\
=& \frac{n-2}{4} (du \otimes dv+dv \otimes du)-\frac{g_2}{2}uv-\lc \frac{n-2}{4} g_1+\frac{n_1-2}{4} g \rc \la \na u, \na v \ra,
\end{align*}
since $u,v$ depend only on $M_1$.

Now, we have
\begin{align*}
Rc_{st}=\widetilde{Rc}_{st} -Rc'(uvg_1),
\end{align*}
which by Lemma \ref{lem:A01} (1) can be simplified as
\begin{align*}
Rc_{st}=&\widetilde{Rc}_{st} + \frac{n_1-2}{2} \na^2(uv)+\frac{1}{2} \Delta(uv) g_1 \\
=& \frac{n+2n_1-6}{4} (du \otimes dv+dv \otimes du)+\lc \frac{6-n}{4}g_1-\frac{n_1-2}{4} g \rc \la \na u, \na v \ra \\
& -\frac{g+g_1}{2}uv+\frac{n_1-2}{2} (u \na^2 v+v \na^2 u).
\end{align*}

Similarly, we define $\widetilde{R}=R(\tilde g)$ and compute from Lemma \ref{lem:A01} (4) that
\begin{align*}
\widetilde{R}_t=-uR_{g+sv g_1}-(n-1) \Delta_{g+sv g_1} u
\end{align*}
and hence
\begin{align*}
\widetilde{R}_{st}=\frac{4-2n-n_1}{2}uv-\frac{(n-1)(n_1-2)}{2}\la \na u, \na v \ra.
\end{align*}
Therefore, we have
\begin{align*}
R_{st}=&\widetilde{R}_{st}-R'(uvg_1)=\widetilde{R}_{st}+\frac{n_1}{2}uv+(n_1-1)\Delta(uv) \\
=& (4-n-2n_1)uv+\frac{5n_1+2n-6-nn_1}{2} \la \na u, \na v \ra.
\end{align*}

Differentiating $2\Delta f-|\na f|^2+R+f-n=\boldsymbol{\mu}$, we obtain
\begin{align*} 
2\Delta f_{st}+2\Delta_t f_s+2\Delta_s f_t-2\la \na f_s,\na f_t \ra+R_{st}+f_{st}=\boldsymbol{\mu}_{st}=0,
\end{align*}
where $f_s=\frac{n_1-2}{2} v$ and $f_t=\frac{n-2}{2} u$. Therefore, we obtain
\begin{align*} 
(\Delta+\frac{1}{2}) f_{st}= \frac{n_1}{2} uv-\frac{3n_1-2}{4} \la \na u, \na v \ra.
\end{align*}

In addition, we compute
\begin{align*}
(\na^2 f)_{st}=&\na^2 f_{st}+\na^2_t f_s+ \na^2_s f_t \\
=& \na^2 f_{st}-\frac{n+n_1-4}{4}(du \otimes dv+dv \otimes du)+\lc \frac{n_1-2}{4} g+\frac{n-2}{4} g_1 \rc \la \na u, \na v \ra.
\end{align*}

Therefore, we have
\begin{align*}
\Phi_{st}=&-R_{st}-(\na^2 f)_{st} \\
=& -\na^2 f_{st}-\frac{n_1-2}{4} (du \otimes dv+dv \otimes du)-g_1 \la \na u, \na v \ra +\frac{g+g_1}{2}uv-\frac{n_1-2}{2} (u \na^2 v+v \na^2 u).
\end{align*}
By taking the trace, we immediately obtain \eqref{E410a} and \eqref{E410b}.
\end{proof}

\begin{prop} \label{prop:401}
Let $(M^n,g)=(M_1\times M_2,g_1 \times g_2 )$ with condition \eqref{E401}. Then any $ h \in \emph{ISD}$ is not integrable of second-order. 
\end{prop}

\begin{proof}
From Lemma \ref{lem:401}, one can represent $h$ as
\begin{align*}
h=u' g_1 + 2 \nabla_1^2 u' + v' g_2,
\end{align*}
for functions $u',v'$ on $M_1$ with $\Delta_1 u'+u'=\Delta_1 v'+v'=0$. By composing with a family of diffeomorphisms generated by $-\na u'$, we may assume $h=u' g_1 + v' g_2=ug+vg_1$, where $u:=v'$ and $v:=u'-v'$.

Suppose $h$ is integrable of second-order, then $\Phi^{(2)}(h,h) \perp \text{ker} (\Phi')$. Notice that by Lemma \ref{lem:401}, $w g,wg_1 \in \text{ker} (\Phi')$ for any $w \in C^{\infty}(M_1)$ with $\Delta_1w+w=0$.

Now, we decompose
\begin{align*}
\int_M \langle \Phi^{(2)}(h,h), w g \rangle &=\int_M \langle \Phi^{(2)} (u g, u g), wg \rangle +\int_M \la \Phi^{(2)}(v g_1, v g_1) , wg \rangle+2 \int_M \la \Phi^{(2)} (u g, v g_1) , wg \rangle \\
&=: I_1+I_2+I_3.
\end{align*}

It follows from Lemma \ref{lem:A02} that
\begin{align*} 
\la \Phi^{(2)} (u g, u g),g \ra=-\Delta f_1+2(n-1)u^2-\frac{3n-2}{2} |\na u|^2,
\end{align*}
where $f_1$ is determined by
\begin{align*}
(\Delta + \frac{1}{2})f_1 = \frac{n}{2} u^2 - \frac{3n-2}{4} |\nabla u|^2.
\end{align*}
Since $\Delta w+w=0$, we have
\begin{align*}
\int_M w f_1 =-2 \int_M w (\Delta + \frac{1}{2})f_1=\int_M -n u^2 w+\frac{3n-2}{2}|\na u|^2 w.
\end{align*}
Therefore, we compute
\begin{align} \label{E402a}
I_1=\int_M w f_1+2(n-1)u^2w-\frac{3n-2}{2} |\na u|^2w=(n-2) \int_M u^2 w.
\end{align}

By similar calculations, we have
\begin{align} \label{E402b}
I_2=\text{Vol}(M_2) \int_{M_1} \la \Phi_1^{(2)}(v g_1, v g_1) , wg_1 \rangle=(4m-2)\int_{M} v^2 w.
\end{align}

Now, it follows from \eqref{E410a} that
\begin{align*} 
\la \Phi^{(2)} (u g, v g_1),g \ra=-\Delta f_2-(6m-1) \la \na u, \na v \ra+(n/2+6m-2) uvw,
\end{align*}
where $f_2$ is determined by
\begin{align*}
(\Delta + \frac{1}{2})f_2 = 2m uv-\frac{6m-1}{2} \la \na u, \na v \ra.
\end{align*}

Therefore, one obtains
\begin{align}\label{E402c}
I_3=2 \int_M w f_2-(6m-1) \la \na u, \na v \ra w+(n/2+6m-2) uvw=(n+4m-4) \int uvw.
\end{align}

Combing \eqref{E402a}, \eqref{E402b} and \eqref{E402c}, we conclude that
\begin{align} \label{E404a}
\int_M \langle \Phi^{(2)}(h,h), w g \rangle=(n-2)\int_{M} u^2 w+2(2m-1)\int_{M} v^2 w+(n+4m-4) \int_{M} uv w.
\end{align}

Next, we decompose
\begin{align*}
\int_M \langle \Phi^{(2)}(h,h), w g_1 \rangle &=\int_M \langle \Phi^{(2)} (u g, u g), w g_1 \rangle +\int_M \la \Phi^{(2)}(v g_1, v g_1) , w g_1 \rangle+2 \int_M \la \Phi^{(2)} (u g, v g_1) , w g_1 \rangle \\
&=: I'_1+I'_2+I'_3.
\end{align*}
It follows from Lemma \ref{lem:A02} that
\begin{align*} 
\la \Phi^{(2)} (u g, u g),g_1 \ra=-\Delta f_1+(n+4m-2)u^2-\frac{8m+n-2}{2} |\na u|^2,
\end{align*}
where we have used the fact that $f_1$ depends only on $M_1$.

Therefore, we have 
\begin{align} 
I'_1=&\int_M w f_1+(n+4m-2) u^2w-\frac{8m+n-2}{2} |\na u|^2w \notag \\
=& \int_M (4m-2) u^2w+(n-4m)|\na u|^2w=\frac{4m+n-4}{2} \int_M u^2w, \label{E403a}
\end{align}
since $2\int_M |\na u|^2w=\int_M u^2w$. Next, we observe that
\begin{align} 
I_2=I_2'=(4m-2)\int_{M} v^2 w.
\label{E403b}
\end{align}
Now, it follows from \eqref{E410b} that
\begin{align}\label{E403c}
I'_3=2 \int_M w f_3-(6m-1) \la \na u, \na v \ra w+2(4m-1) uvw=4(2m-1) \int_M uvw.
\end{align}

Combing \eqref{E403a}, \eqref{E403b} and \eqref{E403c}, we conclude that
\begin{align} \label{E404b}
\int_M \langle \Phi^{(2)}(h,h), w g_1 \rangle=\frac{4m+n-4}{2} \int_M u^2w+2(2m-1)\int_{M} v^2 w+4(2m-1) \int_{M} uv w.
\end{align}

From \eqref{E404a}, \eqref{E404b} and the fact that $u,v,w$ depend only on $M_1$, we conclude that
\begin{align} \label{E404}
\int_{M_1} \psi_1(u,v) w=\int_{M_1} \psi_2(u,v) w=0,
\end{align}
where $\psi_1$ and $\psi_2$ are quadratic functions given by
\begin{align*} 
\psi_1(u,v)=\frac{n-2}{2(2m-1)}u^2+v^2+\frac{n+4m-4}{2(2m-1)} uv, \quad \psi_2(u,v)=\frac{4m+n-4}{4(2m-1)} u^2+v^2+2uv.
\end{align*}

Now, we define constants 
\begin{align*} 
\lambda=\frac{n-4m}{4(2m-1)}, \quad x=\frac{1+\sqrt{1+4\lambda}}{2}, \quad y=\frac{1-\sqrt{1+4\lambda}}{2}.
\end{align*}

Then it is clear that
\begin{align*} 
(xu+v)^2=-\frac{y}{\lambda} \psi_1(u,v)+(1+\frac{y}{\lambda}) \psi_2(u,v)
\end{align*}
and
\begin{align*} 
(yu+v)^2=-\frac{x}{\lambda} \psi_1(u,v)+(1+\frac{x}{\lambda}) \psi_2(u,v).
\end{align*}

Therefore, it follows from \eqref{E404} that for any $w \in C^{\infty}(M_1)$ with $\Delta_1 w+w=0$,
\begin{align*} 
\int_{M_1} (xu+v)^2 w=\int_{M_1} (yu+v)^2w=0
\end{align*}
and hence by Proposition \ref{prop:202} and Lemma \ref{lem:203} that
\begin{align*} 
xu+v=yu+v=0.
\end{align*}
Clearly, it implies that $h=0$, which is a contradiction.

In sum, the proof is complete.
\end{proof}

Combining Proposition \ref{prop:401} with \cite[Theorem 1.1]{LLW21}, we have the following rigidity result.

\begin{thm}
For any compact Einstein manifold $(M_2,g_2)$ satisfying \eqref{E401}, there exists a small constant $\ep>0$ satisfying the following property.

Suppose $(M^{n}, g, f)$ is a Ricci shrinker such that
\begin{align*}
d_{GH} \left\{ (M^{n},g), (\mathbb C \emph{P}^{2m} \times M_2^{n-4m},g_{FS}\times g_2)\right\}<\epsilon,
\end{align*} 
then $(M,g)$ is isometric to $(\mathbb C \emph{P}^{2m} \times M_2^{n-4m},g_{FS}\times g_2)$. 
\end{thm}

In general, we conjecture that the product of two complex projective spaces is also rigid.

\begin{conj}
For any integers $N_1,N_2 \ge 1$, there exists a small constant $\ep=\ep(N_1,N_2)>0$ satisfying the following
property.

Suppose $(M^n, g, f)$ is a Ricci shrinker such that
\begin{align*}
d_{GH} \left\{ (M^n,g), (\mathbb C \emph{P}^{N_1} \times \mathbb C \emph{P}^{N_2},g_{FS}\times g_{FS})\right\}<\epsilon,
\end{align*} 
then $(M^n,g)$ is isometric to $(\mathbb C \emph{P}^{N_1} \times \mathbb C \emph{P}^{N_2},g_{FS}\times g_{FS})$. Here, $n=2(N_1+N_2)$ and $g_{FS} \times g_{FS}$ is the product of the Fubini-Study metrics.
\end{conj}

In addition, motivated by \cite{CM21b}, we also make the following conjecture for the noncompact case.

\begin{conj}
For any positive integers $n$ and $N$, we consider the Ricci shrinker $(\bar M^n, \bar g, \bar f)=(\mathbb C \emph{P}^{N} \times \R^{n-2N},g_{FS} \times g_E,|x|^2/4)$, where $g_E$ denotes the flat metric and $x$ is the coordinate of $\R^{n-2N}$. Then any Ricci shrinker $(M^n, g, f,p)$, where $p$ is a minimum point of $f$, that is sufficiently close to $(\bar M^n, \bar g)$ in the pointed-Gromov-Hausdorff sense must be isometric to $(\bar M^n, \bar g)$.
\end{conj}

%%%% 
\newpage
\appendixpage
\addappheadtotoc
\appendix
\section{The variational formulae} 

We first recall the following formulae of the geometric quantities along conformal deformation, whose proof can be found in \cite[Lemma A.2]{SZ21}. Notice that the sign of the last term in $Rc_{ttt}$ here is different from that in \cite[Lemma A.2]{SZ21}.
\begin{lem} \label{lem:A01}
Let $(M^n,g)$ be a compact Riemannian manifold and $u$ a smooth function on $M$. For $g(t)=(1+tu)g$ and any smooth function $\psi$ on $M$, we have at $t=0$,
\begin{enumerate}[label=(\arabic*)]
\item $\tRc_{t} = -\dfrac{n-2}{2} \nabla^2 u - \dfrac{1}{2}(\Delta u) g$.
\item $\tRc_{tt} = (n-2) u \nabla^2 u + \dfrac{3(n-2)}{2} du \otimes du + \lc u \Delta u - \dfrac{n-4}{2}|\nabla u |^2 \rc g$.
\item $\tRc_{ttt} = -3(n-2) u^2 \nabla^2 u - 9(n-2) u du \otimes du - \lc 3 u^2 \Delta u - 3(n-4)u|\nabla u|^2 \rc g $.
\item $R_{t} = -u R -(n-1) \Delta u$.
\item $R_{tt} = 2 u^2 R + 4(n-1) u \Delta u - \dfrac{(n-1)(n-6)}{2} |\nabla u|^2$.
\item $R_{ttt} = - 6 u^3 R - 18(n-1) u^2 \Delta u + \dfrac{9(n-1)(n-6)}{2}u|\na u|^2$.
\item $\nabla^2_t \psi = \dfrac{1}{2} \lc - du \otimes d \psi - d \psi \otimes d u + \la \nabla \psi, \nabla u \ra g \rc$.
\item $\nabla^2_{tt} \psi = u \lc du \otimes d \psi + d \psi \otimes du -\la \nabla \psi, \nabla u \ra g \rc$.
%\item $\nabla^2_{ttt} \psi = 3 u^2 \lc -du \otimes d \psi - d \psi \otimes du + \la \nabla \psi, \nabla u \ra g \rc$.
\item $\Delta_t \psi= -u \Delta \psi + \dfrac{n-2}{2} \la \nabla u , \nabla \psi \ra$.
\item $\Delta_{tt} \psi = 2 u^2 \Delta \psi -2(n-2)u \la \nabla u , \nabla \psi \ra $.
%\item $\Delta_{ttt} \psi = -6 u^3 \Delta \psi -9(n-2) u^2 \la \nabla u ,\nabla \psi \ra$.
%\item $Rc_{st}=\dfrac{n-2}{2} (u \na^2v+v\na^2u)+\dfrac{3(n-2)}{4} \lc du \otimes dv+dv\otimes du \rc+\dfrac{1}{2} \lc u\Delta v+v\Delta u-(n-4)\la \na u, \na v \ra\rc g$.
%\item $R_{st}=2uvR+2(n-1) (u\Delta v+v\Delta u)-\dfrac{(n-1)(n-6)}{2} \la \na u,\na v\ra$.
%\item $\na^2_{st} \psi=\dfrac{1}{2}\lc d(uv) \otimes d\psi+d\psi \otimes d(uv)-\la \na \psi, \na (uv) \ra g \rc$.
\end{enumerate}
\end{lem}

Let $(M^n,g)$ be a compact Einstein manifold with Einstein constant $1/2$. 

\begin{lem} \label{lem:A02}
For any smooth function $u$ with $\Delta u+u=0$, we set $g(t)=(1+tu)g$ and $f(t)=f(g(t))$. At $t=0$, we have the following variational formulae.
\begin{enumerate}[label=(\roman*)]
%\item $\Phi_t=\boldsymbol{\mu}_{t} = \boldsymbol{\mu}_{tt} = 0$ and $f_t= \frac{n-2}{2} u $.

\item (The second-order variational formula) 
\begin{align}\label{EA01}
\Phi_{tt} = -\nabla^2 f_{tt} - \frac{n-2}{2}du \otimes du+(u^2- |\nabla u|^2) g -(n-2)u \nabla^2 u , 
\end{align}
where $f_{tt}$ is determined by
\begin{align} \label{EA02}
(\Delta + \frac{1}{2})f_{tt} = \frac{n}{2} u^2 - \frac{3n-2}{4} |\nabla u|^2.
\end{align}

\item (The third-order variational formula) 
\begin{align}\label{EA03}
\Phi_{ttt} & = -\nabla^2 f_{ttt} + \frac{3}{2} \lc du \otimes df_{tt} + d f_{tt} \otimes du - \la \nabla f_{tt} ,\nabla u \ra g \rc \notag \\
& + 6(n-2) u du \otimes du + 3(n-2) u^2 \nabla^2 u -3 u^3 g - \frac{3(n-6)}{2} u |\nabla u|^2 g,
\end{align}
where $f_{ttt}$ is determined by
\begin{align}\label{EA03a}
(\Delta + \frac{1}{2})f_{ttt} = 3u \Delta f_{tt}  - \frac{3(3n-2)}{2} u^3 + \frac{9(3n-2)}{4} u |\nabla u|^2+\frac{n-2}{2} \aint u^3.
\end{align}
\end{enumerate}
\end{lem}
\begin{proof}
(i) By our assumption of $u$, we know that $\Phi'(ug)=0$ and hence $\Phi_t= \boldsymbol{\mu}_{tt} = 0$, where $\boldsymbol{\mu}=\boldsymbol{\mu}(g(t),1)$. By differentiating twice for the Euler-Lagrange equation $2\Delta f-|\na f|^2+R+f-n=\boldsymbol{\mu}$, we obtain
\begin{align} \label{EA01a}
2\Delta f_{tt}+4\Delta_t f_t-2|\na f_t|^2+R_{tt}+f_{tt}=0, 
\end{align}
where we have used the fact that $f(0)$ is a constant. On the one hand, it follows from \eqref{E202b} that
\begin{align*}
(\Delta+\frac 1 2)(nu-2f_t)=\delta^2(ug)=\Delta u=-u
\end{align*}
and hence $f_t= \frac{n-2}{2} u $. On the other hand, it follows from Lemma \ref{lem:A01} that
\begin{align*}
R_{tt}=(4-3n)u^2- \frac{(n-1)(n-6)}{2} |\nabla u|^2
\end{align*}
and 
\begin{align*}
\Delta_t f_t=-u \Delta f_t +\frac{n-2}{2} \la \nabla u , \nabla f_t \ra=\frac{n-2}{2}u^2+\frac{(n-2)^2}{4}|\na u|^2.
\end{align*}
From \eqref{EA01a}, we immediately obtain \eqref{EA02}.
Next, we compute
\begin{align*}
\Phi_{tt}=-Rc_{tt}-(\na^2 f)_{tt}=-Rc_{tt}-\na^2 f_{tt}-2\na_t^2 f_t.
\end{align*}
From Lemma \ref{lem:A01}(2)(7), we obtain \eqref{EA01} after simplification.

(ii) By differentiating three times for the Euler-Lagrange equation $2\Delta f-|\na f|^2+R+f-n=\boldsymbol{\mu}$, we obtain
\begin{align} \label{EA04}
2 \Delta f_{ttt}+6\Delta_t f_{tt}+6\Delta_{tt}f_t+6u|\na f_t|^2-6 \la \na f_t,\na f_{tt} \ra+R_{ttt}+f_{ttt}=\boldsymbol{\mu}_{ttt}.
\end{align}
From the definition of $\boldsymbol{\mu}$, we have
\begin{align*}
\boldsymbol{\mu}_{ttt}=\aint \la \Phi_{tt},ug \ra=(n-2) \aint u^3,
\end{align*}
where the last equality can be calculated similarly as \cite[Theorem 5.7]{Kr16} by using \eqref{EA01} and \eqref{EA02}, see also \cite[Proposition 9.1]{Kr20}. From Lemma \ref{lem:A01}(6)(9)(10), we obtain \eqref{EA03a}.

Similarly, we compute
\begin{align*}
\Phi_{ttt}=-Rc_{ttt}-(\na^2 f)_{ttt}=-Rc_{ttt}-\na^2 f_{ttt}-3\na_t^2 f_{tt}-3\na_{tt}^2 f_t.
\end{align*}
Therefore, \eqref{EA03} follows from Lemma \ref{lem:A01}(3)(7)(8) with a routined simplification.
\end{proof}

\begin{lem} \label{lem:A03}
For any smooth functions $u,v$ with $\Delta u+u=\Delta v+v=0$, we set $g(t,s)=(1+tu+sv)g$ and $f(t,s)=f(g(t,s))$. At $t=s=0$, we have the following variational formula.
\begin{align*} 
\Phi_{st}=&-\na^2 f_{st}-\frac{n-2}{4}(du \otimes dv+dv \otimes du)+(uv-\la \na u,\na v \ra) g-\frac{n-2}{2}(u\na^2v+v\na^2 u),
\end{align*} 
where $f_{st}$ is determined by 
\begin{align}
(\Delta+\frac{1}{2})f_{st}=\frac{n}{2} uv-\frac{3n-2}{4} \la \na u,\na v \ra. \label{EA06}
\end{align} 
In particular, 
\begin{align}
\la \Phi_{st},g \ra=-\Delta f_{st}+\frac{2-3n}{2} \la \na u, \na v \ra+2(n-1)uv. \label{EA07}
\end{align}
\end{lem}
\begin{proof}
The proof follows from Lemma \ref{lem:A02} (i) by polarization.
\end{proof}

\begin{lem} \label{lem:A04}
For any smooth function $u$ with $\Delta u+u=0$ and symmetric $2$-tensor $h$ with $\delta h=0$, we set $g(t,s)=(1+tu)g+sh$ and $f(t,s)=f(g(t,s))$. At $t=s=0$, we have the following variational formulae.
\begin{enumerate}[label=(\roman*)]
\item (Cross second-order variational formula for $Rc$)
\begin{align}
Rc_{st}=&\frac{n-2}{2}C_{ij}^k u_k+\frac{u}{2}h+\la h,\na^2 u \ra\frac{g}{2}-\la \na H,\na u \ra \frac{g}{4} \notag \\
&+\frac{1}{2} \lc \Delta(uh)+2Rm(uh)-uh+\na^2(uH)+ 2\delta^* \delta(uh)\rc. \label{EA09aa}
\end{align} 
\item (Cross second-order variational formula for $\Phi$)
\begin{align} 
\Phi_{st}=-Rc_{st}-\na^2 f_{st}+\frac{n-2}{2}C_{ij}^k u_k+\frac{1}{4}\lc du \otimes dH+dH \otimes du-\la \na u,\na H \ra g \rc, \label{EA09a}
\end{align} 
where
and $f_{st}$ is determined by 
\begin{align}
(\Delta+\frac{1}{2})f_{st}=-\frac{1}{2}u\Delta H-\frac{3}{4} \la \na u, \na H \ra. \label{EA09b}
\end{align}
\end{enumerate} 
Here, $H=\emph{Tr}(h)$ and $C_{ij}^k=\frac{1}{2} g^{kl} (\na_i h_{jl}+\na_j h_{il}-\na_l h_{ij})$. In particular, 
\begin{align}
\la \Phi_{st},g \ra=-\frac{n-2}{2}\la h,\na^2 u \ra-\frac{3}{2} \la \na H, \na u \ra-u\Delta H+\frac{uH}{2}-\Delta f_{st}. \label{EA09c}
\end{align}
\end{lem}

\begin{proof}
We define $\tilde g(t,s)=(1+tu)(g+sh)$ and $\widetilde{Rc}=Rc(\tilde g)$.

From Lemma \ref{lem:A01}(1), we have at $t=0$,
\begin{align*}
\widetilde{Rc}_t=-\frac{n-2}{2}\na^2_{g+sh} u - \frac{1}{2} (\Delta_{g+sh} u)(g+sh).
\end{align*}
By taking the derivative for $s$, we obtain
\begin{align*}
\widetilde{Rc}_{st} 
= \frac{n-2}{2} C_{ij}^{k} u_k + \frac{1}{2}uh + \la h,\nabla^2 u \ra \frac{g}{2} - \la \nabla H,\nabla u \ra \frac{g}{4} 
\end{align*}
where we have used
\begin{align*}
(\na_i \na_j u)_s=-(\Gamma_{ij}^k)_s u_k=-C_{ij}^k u_k.
\end{align*}
and
\begin{align*}
(\Delta_{g+sh} u)_s=-\la h,\na^2 u \ra-\la \delta h, \na u\ra+\frac{1}{2}\la \na H, \na u \ra=-\la h,\na^2 u \ra+\frac{1}{2}\la \na H, \na u \ra.
\end{align*}

Therefore, \eqref{EA09aa} follows from the fact that
\begin{align*}
Rc_{st}=\widetilde{Rc}_{st} -Rc'(uh).
\end{align*}
and by \cite[Theorem 1.174(d)]{Be87},
\begin{align*}
-2Rc'(uh)=\Delta(uh)+2Rm(uh)-uh+\na^2(uH)+2\delta^* \delta(uh).
\end{align*}

Similarly, we define $\widetilde{R}=R(\tilde g)$ and compute from Lemma \ref{lem:A01} (4) that
\begin{align*}
\widetilde{R}_t=-uR_{g+sh}-(n-1) \Delta_{g+sh} u
\end{align*}
and hence
\begin{align*}
\widetilde{R}_{st}=u(\Delta H+\frac{H}{2})+(n-1) \la \na^2 u, h \ra-\frac{n-1}{2} \la \na H, \na u \ra,
\end{align*}
where we have used \cite[Theorem 1.174(e)]{Be87}.

Therefore, we have
\begin{align*}
R_{st}=&\widetilde{R}_{st}-R'(uh)=\widetilde{R}_{st}+\Delta(uH)-\delta^2(uh)+\frac{uH}{2} \\
=& \widetilde{R}_{st}+u\Delta H+2\la \na H, \na u \ra-\la h, \na^2 u \ra-\frac{uH}{2}\\
=& (n-2)\la h,\na^2 u \ra+\frac{5-n}{2} \la \na H, \na u \ra+2u\Delta H.
\end{align*}

Differentiating $2\Delta f-|\na f|^2+R+f-n=\boldsymbol{\mu}$, where $\boldsymbol{\mu}=\boldsymbol{\mu}(g(t,s),1)$, we obtain
\begin{align*} 
2\Delta f_{st}+2\Delta_t f_s+2\Delta_s f_t-2\la \na f_s,\na f_t \ra+R_{st}+f_{st}=\boldsymbol{\mu}_{st}=0,
\end{align*}
where the last equality holds since $ug \in \text{ker}(\Phi')$. From Lemma \ref{lem:201}, $f_t=\frac{n-2}{2} u$, $f_s=H/2$ and hence by direct calculations we obtain \eqref{EA09b}.

In addition, we compute
\begin{align*}
(\na^2 f)_{st}=&\na^2 f_{st}+\na^2_t f_s+ \na^2_s f_t \\
=& \na^2 f_{st}+\frac{1}{4}(-du \otimes dH-dH \otimes du+\la \na H, \na u \ra g)-\frac{n-2}{2}C^k_{ij} u_k.
\end{align*}

From $\Phi_{st}=-Rc_{st}-(\na^2 f)_{st}$, we immediately obtain \eqref{EA09a} and hence \eqref{EA09c} follows by taking the trace and noting that $\la \delta^* \delta(uh),g\ra=-\la h,\na^2 u \ra$ since $\delta h=0$.
\end{proof}

\vskip10pt

Yu Li, Institute of Geometry and Physics, University of Science and Technology of China, No. 96 Jinzhai Road, Hefei, Anhui Province, 230026, China; yuli21@ustc.edu.cn.\\

Wenjia Zhang, School of Mathematical Sciences, University of Science and Technology of China, No. 96 Jinzhai Road, Hefei, Anhui Province, 230026, China; wj12345678@mail.ustc.edu.cn.\\

\end{document}